\documentclass{amsart}

\makeatletter
\@namedef{subjclassname@2020}{%
  \textup{2020} Mathematics Subject Classification}
\makeatother % Cambiar el año de la MSC, de 2010 a 2020

\usepackage{amsthm,amssymb,amsfonts,latexsym,mathtools,thmtools}
\usepackage{amsmath}
\usepackage[T1]{fontenc}
\usepackage{tikz-cd} % Commutative diagrams.
\usepackage{enumitem} % Customization of items.
\usepackage{hyperref} %hyperlinks
\hypersetup{
    colorlinks=true,
    linkcolor=blue,
    filecolor=blue,
    urlcolor=cyan,
    linktocpage=true
}

\newtheorem{theorem}{Theorem}[section]
\newtheorem{lemma}[theorem]{Lemma}

\theoremstyle{definition}
\newtheorem{definition}[theorem]{Definition}
\newtheorem{proposition}[theorem]{Proposition}
\newtheorem{example}[theorem]{Example}
\newtheorem{remark}[theorem]{Remark}
\newtheorem{corollary}[theorem]{Corollary}

\numberwithin{equation}{section}

\begin{document}

\title[$\Sigma$-semicommutative rings and their skew PBW extensions]{$\Sigma$-semicommutative rings and their skew PBW extensions}

%    Remove any unused author tags.

%    author one information

\author{H\'ector Su\'arez}
\address{Universidad Pedag\'ogica y Tecnol\'ogica de Colombia - Sede Tunja}
\curraddr{Campus Universitario}
\email{hector.suarez@uptc.edu.co}
\thanks{}

\author{Armando Reyes}
\address{Universidad Nacional de Colombia - Sede Bogot\'a}
\curraddr{Campus Universitario}
\email{mareyesv@unal.edu.co}
\thanks{}

\thanks{The authors were supported by the research fund of Faculty of Science, Code HERMES 52464, Universidad Nacional de Colombia - Sede Bogot\'a, Colombia.}



\date{}

\dedicatory{Dedicated to the memory of Professor V. A. Artamonov}

\begin{abstract}

In this paper, we introduce the concept of $\Sigma$-semicommutative ring, for $\Sigma$ a finite family of endomorphisms of a ring $R$. We relate this class of rings with other classes of rings such that Abelian, reduced, $\Sigma$-rigid, nil-reversible and rings satisfying the $\Sigma$-skew reflexive nilpotent property. Also, we study some ring-theoretical properties of skew PBW extensions over $\Sigma$-semicommutative rings. We prove that if a ring $R$ is $\Sigma$-semicommutative with certain conditions of compatibility on derivations, then for every skew PBW extension $A$ over $R$, $R$ is Baer if and only if $R$ is quasi-Baer, and equivalently, $A$ is quasi-Baer if and only if $A$ is Baer. Finally, we consider some topological conditions for skew PBW extensions over $\Sigma$-semicommutative rings.

\end{abstract}

\maketitle

%\tableofcontents

%%%%%%%%%%%%%%%%%%%%%%%%%%%%%%%%%%%%%
%%%%%%%%%%%%%%%%%%%%%%%%%%%%%%%%%%%%%

\section{Introduction}\label{section}

Throughout this paper, every ring is associative with identity
unless otherwise stated. For a ring $R$, $N(R)$ denotes its set of
nilpotent elements, $N_{*}(R)$ is its prime radical, $N^{*}(R)$ is
its upper nilradical (i.e., the sum of all nil ideals), and $J(R)$
denotes its Jacobson radical. It is well-known that
$N_{*}(R)\subseteq N^{*}(R)\subseteq N(R)$. Due to Marks
\cite{Marks2001}, a ring $R$ is called \emph{NI} if $N^{*}(R)=N(R)$
(equivalently, $N(R)$ is an ideal of $R$). According to Birkenmeier
et al. \cite{Birkenmeier1993}, a ring $R$ is called
2-\textit{primal} if $N(R)=N_{*}(R)$ (Sun \cite{Sun1991} used the
term \emph{weakly symmetric} for these rings). A ring $R$ is called
\emph{NJ} if $J(R)=N(R)$ (Jiang et al. \cite[p. 45]{Jiang2019}). It
is clear that NJ and 2-primal rings are NI. However, in
\cite[Example 2.2]{Jiang2019} an example of a ring that is 2-primal
but not NJ is presented, and we can also find an NJ ring that it is
not 2-primal \cite[Example 1.2]{Hwang2006}. Shin \cite{Shin1973} and
Sun \cite{Sun1991} studied some topological conditions for 2-primal
rings. Hwang et al. \cite{Hwang2006} and Jiang et al.
\cite{Jiang2019} studied topological conditions for NI and NJ rings,
respectively.

\vspace{0.2cm}

Recall that a ring is \emph{reduced} if it has no nonzero nilpotent
elements. Lambek \cite{Lambek1971} called a ring $R$
\emph{symmetric} provided $abc = 0$ implies $acb = 0$ for $a, b,
c\in R$. Ouyang and Chen \cite[Definition 1]{Ouyang2010} introduced
the notion of weak symmetric ring in the following way: a ring $R$
is called a \emph{weak symmetric} ring if $abc\in N(R)$ implies
$acb\in N(R)$, for every elements $a, b, c\in R$. They proved that
this notion extends the concept of symmetric ring \cite[Proposition
2.1]{Ouyang2010}. However the converse of the assertion is false,
i.e., there exists a weak symmetric ring which is not symmetric
\cite[Example 2.2]{Ouyang2010}. Habeb \cite{Habeb1970} called a ring
$R$ \emph{zero commutative} if for $a, b\in R$, $ab = 0$ implies $ba
= 0$ (Cohn \cite{Cohn1999} used the term \emph{reversible} for what
is called zero commutative). A generalization of reversible rings is
the notion of semicommutative ring. A ring $R$ is
\emph{semicommutative} if $ab = 0$ implies $aRb = 0$ for $a, b\in R$
(Shin \cite{Shin1973} and Bell \cite{Bell1970} used the terms
\emph{SI} and \emph{Insertion-of-Factors-Property} ({\em IFP}) for
the semicommutative rings, respectively). Shin \cite[Lemma 1.1 and
1.2]{Shin1973} proved that every reduced ring is symmetric but the
converse does not hold \cite[Example II.5]{Anderson1999}, and that
$R$ is semicommutative if and only if $r_R(a)=\{s\in R\mid as=0\}$
is an ideal of $R$ if and only if any annihilator right (left) ideal
of $R$ is an ideal of $R$. He also proved that any symmetric ring is
semicommutative \cite[Proposition 1.4]{Shin1973} but the converse
does not hold \cite[Example 5.4(a)]{Shin1973}. Chaturvedi et al.
\cite{Chaturvedi2021} introduced the concept of Z-symmetric rings
and proved that semicommutative rings are Z-symmetric rings. Another
generalization of reversible rings are the reflexive rings. A ring
$R$ is called \emph{reflexive} if $aRb = 0$ implies $bRa = 0$, for
$a, b\in R$. Notice that there exists a reflexive and
semmicomutative ring which is not symmetric \cite[Examples 5 and
7]{Marks2002}. Kwak and Lee proved that a ring $R$ is reflexive and
semicommutative if and only if $R$ is reversible \cite[Proposition
2.2]{KwakLee2012}.

\vspace{0.2cm}

A ring $R$ is called \emph{Abelian} if all its idempotents are
central. Notice that semicommutative rings are Abelian \cite[Lemma
2.7]{Shin1973}, and every semicommutative ring is 2-primal
\cite[Theorem 1.5]{Shin1973}. Kaplansky \cite{Kaplansky1968}
introduced the concept of \emph{Baer ring} as rings in which the
right (left) annihilator of every nonempty subset is generated by an
idempotent. According to Clark \cite{Clark1967}, a ring $R$ is
called \emph{quasi-Baer} if the right annihilator of each right
ideal of $R$ is generated (as a right ideal) by an idempotent. These
two concepts are left-right symmetric.

\vspace{0.2cm}

Krempa \cite{Krempa1996} introduced the notion of rigidness as a
generalization of reduced rings. An endomorphism $\sigma$ of a ring
$R$ is called \textit{rigid} if for $a \in R$, $a\sigma(a) = 0$
implies $a = 0$, and a ring $R$ is called \textit{$\sigma$-rigid} if
there exists a rigid endomorphism $\sigma$ of $R$. With the aim of
extending $\sigma$-rigid rings and semicommutative rings, Ba\c{s}er
et al. \cite{Baser} introduced the notion of a
$\sigma$-semicommutative ring with the endomorphism $\sigma$ of a
ring $R$ in the following way: an endomorphism $\sigma$  of a ring
$R$ is called \emph{semicommutative} if whenever $ab = 0$ for $a,
b\in R$, $aR\sigma(b) = 0$. A ring $R$ is called
\emph{$\sigma$-semicommutative} if there exists a semicommutative
endomorphism $\sigma$ of $R$ \cite[Definition 2.1]{Baser}. They
studied characterizations of $\sigma$-semicommutative rings and
their related properties.

\vspace{0.2cm}

Regarding the objects of interest in this article, the {\em skew PBW
extensions}, these objects were defined by Gallego and Lezama
\cite{LezamaGallego} with the aim of generalizing families of
noncommutative rings as PBW extensions introduced by Bell and
Goodearl \cite{BellGoodearl1988}, skew polynomial rings (of
injective type) defined by Ore \cite{Ore1933}, and other as solvable
polynomial rings, diffusion algebras, some types of
Auslander-Gorenstein rings, some Calabi-Yau and skew Calabi-Yau
algebras, some Artin-Schelter regular algebras, some Koszul algebras
(see \cite{Fajardoetal2020} or \cite{ReyesSuarez2018-3} for a
detailed reference to each of these families). Several ring and
theoretical properties of skew PBW extensions have been investigated
by some people (e.g., Artamonov \cite{Artamonov2015}, Fajardo et al.
\cite{Fajardoetal2020}, Hamidizadeh et al.
\cite{Hamidizadehetal2020}, Hashemi et al.
\cite{HashemiKhalilAlhevaz2017},  Lezama \cite{Lezama2020}, and
Louzari et al. \cite{LouzariReyes2020}). Precisely, the authors have
studied the ring-theoretical notions mentioned above in the setting
of skew PBW extensions. For example, in \cite{ReyesSuarez2018-3}
they investigated these objects over Armendariz rings, and
considered the properties of Abelian, rigidness, Baer, quasi-Baer,
p.p., and p.q.-Baer (c.f. \cite{ReyesSuarez2019Radicals} for the
characterization of several radicals of these noncommutative rings),
while in \cite{ReyesSuarez2021Discr} they characterized skew PBW
extensions over several weak symmetric rings (a generalization of
symmetric rings), and recently, in \cite{SuarezChaconReyes2021} the
authors established necessary or sufficient conditions to guarantee
that skew PBW extensions are NI or NJ rings.

\vspace{0.2cm}

Motivated by these previous works, our aim in this paper is to
investigate the property of semicommutativity over skew PBW
extensions, and consider its relationship with other classes of
rings such that Abelian, reduced, $\Sigma$-rigid, nil-reversible and
rings satisfying the $\Sigma$-skew reflexive nilpotent property. The
important results about these topics are presented in Theorems
\ref{teo-Sigmarigidiffreduc}, \ref{teo-SimasemiciffSigmaskewRNP},
\ref{teo-compilSigma-semic}, \ref{teo-equivBaer}, and
\ref{teo-RSigmasemisiiASigmasemi}, Propositions
\ref{prop-AsemimplRSigmasem}, \ref{prop-anulad}, and Corollary
\ref{cor-relatrigid}. We also consider some topological conditions
for skew PBW extensions over $\Sigma$-semicommutative rings, and the
original results of the paper are formulated in Propositions
\ref{prop-generteo4.2Shin}, \ref{prop-generteo2.3Sun},
\ref{prop-generteo3.7Hwang}, and \ref{prop-generteo4.10Jian}. In
this way, our paper can be considered as a contribution to the study
of ring-theoretical and topological properties of families of
noncommutative rings defined by generators and relations between
them.

\vspace{0.2cm}

Throughout the paper, $\mathbb{N}$ and $\mathbb{Z}$ denote the set
of natural numbers including zero and the set of integer numbers,
respectively.

\section{Skew PBW extensions}
We start by recalling the definition of skew PBW extension
introduced by Gallego and Lezama \cite{LezamaGallego}.
%%%%%%%%%%%%%%%%%%%%%%%%
\begin{definition}\cite[Definition 1]{LezamaGallego} \label{def.skewpbwextensions}
Let $R$ and $A$ be rings. We say that $A$ is a \textit{skew PBW
extension over} $R$, denoted $A=\sigma(R)\langle
x_1,\dots,x_n\rangle$, if the following conditions hold:
\begin{enumerate}
\item[\rm (i)]$R$ is a subring of $A$ sharing the same identity element.
\item[\rm (ii)] there exist finitely many elements $x_1,\dots ,x_n\in A$ such that $A$ is a left free $R$-module, with basis the set of standard monomials
\begin{center}
${\rm Mon}(A):= \{x^{\alpha}:=x_1^{\alpha_1}\cdots
x_n^{\alpha_n}\mid \alpha=(\alpha_1,\dots ,\alpha_n)\in
\mathbb{N}^n\}$.
\end{center}
Moreover, $x^0_1\cdots x^0_n := 1 \in {\rm Mon}(A)$.
\item[\rm (iii)]For each $1\leq i\leq n$ and any $r\in R\ \backslash\ \{0\}$, there exists an element $c_{i,r}\in R\ \backslash\ \{0\}$ such that $x_ir-c_{i,r}x_i\in R$.
\item[\rm (iv)]For $1\leq i,j\leq n$, there exists an element $d_{i,j}\in R\ \backslash\ \{0\}$ such that
\[
x_jx_i-d_{i,j}x_ix_j\in R+Rx_1+\cdots +Rx_n,
\]
that is, there exist elements $r_l^{(i,j)} \in R$ with
\[
x_jx_i - d_{i,j}x_ix_j = \sum_{l=0}^{n} r_l^{(i,j)}x_n\ \ \
(x_0:=1).
\]
\end{enumerate}
\end{definition}
From the definition, every non-zero element $f \in A$ can be
expressed uniquely as $f = a_0 + a_1X_1 + \cdots + a_mX_m$, with
$a_i \in R$ and $X_i \in \text{Mon}(A)$, $0 \leq i \leq m$
\cite[Remark 2]{LezamaGallego}. For $X = x^{\alpha} =
x_1^{\alpha_1}\cdots x_n^{\alpha_n} \in \text{Mon}(A)$,
$\text{deg}(X) = |\alpha| := \alpha_1 + \cdots + \alpha_n$.

\vspace{0.2cm}

The following result establishes explicitly the relation between
skew PBW extensions and skew polynomial rings introduced by Ore
\cite{Ore1933}.
%%%%%%%%%%%%%%%%%%%%%%%%
\begin{proposition}\cite[Proposition 3]{LezamaGallego} \label{sigmadefinition}
If $A=\sigma(R)\langle x_1,\dots,x_n\rangle$ is a skew PBW
extension, then for each $1\leq i\leq n$ there exist an injective
endomorphism $\sigma_i:R\rightarrow R$ and a $\sigma_i$-derivation
$\delta_i:R\rightarrow R$ such that
$x_ir=\sigma_i(r)x_i+\delta_i(r)$, for every $r\in R$.
\end{proposition}
%%%%%%%%%%%%%%%%%%%%%%%%
\begin{remark}\label{rem-notacionSigmaDelta}
We use the notation $\Sigma:=\{\sigma_1,\dots,\sigma_n\}$ and
$\Delta:=\{\delta_1,\dots,\delta_n\}$ for the families of
monomorphisms $\sigma_i$'s and $\sigma_i$-derivations $\delta_i$'s,
respectively, formulated in Proposition \ref{sigmadefinition}. For a
skew PBW extension $A = \sigma(R)\langle x_1,\dotsc, x_n\rangle$, we
say that $(\Sigma, \Delta)$ is a \textit{system of endomorphisms and
$\Sigma$-derivations} of $R$ with respect to $A$. For $\alpha =
(\alpha_1, \dots , \alpha_n) \in \mathbb{N}^n$, $\sigma^{\alpha}:=
\sigma_1^{\alpha_1}\circ \cdots \circ \sigma_n^{\alpha_n}$,
$\delta^{\alpha} := \delta_1^{\alpha_1} \circ \cdots \circ
\delta_n^{\alpha_n}$, where $\circ$ denotes the classical
composition of functions. The zero element of $\mathbb{N}^n$ is
$0:=(0,\dots,0)$.
\end{remark}
%%%%%%%%%%%%%%%%%%%%%%%%
\begin{definition} \cite[Definition 4]{LezamaGallego} \label{quasicommutative}
Let $A=\sigma(R)\langle x_1,\dots,x_n\rangle$ be a skew PBW
extension over $R$.
\begin{itemize}
    \item[{\rm (i)}] $A$ is called {\it quasi-commutative} if the conditions (iii) and (iv) presented in the Definition \ref{def.skewpbwextensions} are replaced by the following:
\begin{enumerate}
    \item[(iii')] For every $1 \leq i \leq n$ and $r \in R \setminus \left \{0 \right \}$ there exists $c_{i,r} \in R \setminus \left \{0 \right \}$ such that $x_ir = c_{i,r}x_i$.
\item[(iv')] For every $1 \leq i, j \leq n$ there exists $d_{i,j} \in R \setminus \left \{0 \right \}$ such that $x_jx_i = d_{i,j} x_ix_j$.
\end{enumerate}
    \item[(ii)] $A$ is said to be {\it bijective} if  $\sigma_i$ is bijective for each $1 \leq i \leq n$, and $d_{i,j}$ is invertible for any $1 \leq i <j \leq n$.
\end{itemize}
\end{definition}
%%%%%%%%%%%%%%%%%%%%%%%%
If $\sigma_i$ is the identity map of $R$ for each $i = 1, \dotsc,
n$, (we write $\sigma_i = {\rm id}_R$), we say that $A$ is a skew
PBW extension of \textit{derivation type}. Similarly, if $\delta_i =
0$ for every $i$, then $A$ is called a skew PBW extension of
\textit{endomorphism type} \cite[Definition 2.3]{LezamaAcostaReyes}.
%%%%%%%%%%%%%%%%%%%%%%%%
\begin{example}
A great variety of algebras can be expressed as skew PBW extensions.
For example, PBW extensions defined by Bell and Goodearl
\cite{BellGoodearl1988} (which include families of noncommutative
rings as enveloping algebras of finite dimensional Lie algebras and
differential operators), Weyl algebras, skew polynomial rings of
injective type, some types of Auslander-Gorenstein rings, some skew
Calabi-Yau algebras, examples of quantum polynomials, some quantum
universal enveloping algebras, and other algebras of great interest
in noncommutative algebraic geometry and noncommutative differential
geometry such as Auslander-regular algebras. For more details about
these examples, see \cite{Fajardoetal2020, GomezSuarez2019,
Lezama2020, Lezama2021, LezamaReyes2014, LouzariReyes2020,
Suarez2017, SuarezLezamaReyes2017}.
\end{example}
%%%%%%%%%%%%%%%%%%%%%%%%
The following result is useful when one need some computations with
elements of skew PBW extensions.
%%%%%%%%%%%%%%%%%%%%%%%%
\begin{remark}\cite[Remark 6.1.8
(iii)]{Fajardoetal2020}\label{remark-remark2.8} If
$A=\sigma(R)\langle
    x_1,\dots,x_n\rangle$ is a skew PBW extension over $R$, $r$ is an element of $R$, and $\alpha=(\alpha_1,\dotsc, \alpha_n)\in \mathbb{N}^{n}$, then
        {\small{\begin{align*}
                x^{\alpha}r = &\ x_1^{\alpha_1}x_2^{\alpha_2}\dotsb x_{n-1}^{\alpha_{n-1}}x_n^{\alpha_n}r = x_1^{\alpha_1}\dotsb x_{n-1}^{\alpha_{n-1}}\biggl(\sum_{j=1}^{\alpha_n}x_n^{\alpha_{n}-j}\delta_n(\sigma_n^{j-1}(r))x_n^{j-1}\biggr)\\
                + &\ x_1^{\alpha_1}\dotsb x_{n-2}^{\alpha_{n-2}}\biggl(\sum_{j=1}^{\alpha_{n-1}}x_{n-1}^{\alpha_{n-1}-j}\delta_{n-1}(\sigma_{n-1}^{j-1}(\sigma_n^{\alpha_n}(r)))x_{n-1}^{j-1}\biggr)x_n^{\alpha_n}\\
                + &\ x_1^{\alpha_1}\dotsb x_{n-3}^{\alpha_{n-3}}\biggl(\sum_{j=1}^{\alpha_{n-2}} x_{n-2}^{\alpha_{n-2}-j}\delta_{n-2}(\sigma_{n-2}^{j-1}(\sigma_{n-1}^{\alpha_{n-1}}(\sigma_n^{\alpha_n}(r))))x_{n-2}^{j-1}\biggr)x_{n-1}^{\alpha_{n-1}}x_n^{\alpha_n} + \dotsb\\
                + &\ x_1^{\alpha_1}\biggl( \sum_{j=1}^{\alpha_2}x_2^{\alpha_2-j}\delta_2(\sigma_2^{j-1}(\sigma_3^{\alpha_3}(\sigma_4^{\alpha_4}(\dotsb (\sigma_n^{\alpha_n}(r))))))x_2^{j-1}\biggr)x_3^{\alpha_3}x_4^{\alpha_4}\dotsb x_{n-1}^{\alpha_{n-1}}x_n^{\alpha_n} \\
                + &\ \sigma_1^{\alpha_1}(\sigma_2^{\alpha_2}(\dotsb
                (\sigma_n^{\alpha_n}(r))))x_1^{\alpha_1}\dotsb x_n^{\alpha_n}, \ \ \
                \ \ \ \ \ \ \ \sigma_j^{0}:={\rm id}_R\ \ {\rm for}\ \ 1\le j\le n.
                \end{align*}}}
\end{remark}
%%%%%%%%%%%%%%%%%%%%%%%%
For the purposes of the paper, we need to establish a criterion
which allows us to extend the family $\Sigma$ of injective
endomorphisms and the family of $\Sigma$-derivations $\Delta$ of the
ring $R$ to the skew PBW extension $A$. For the next result,
consider $\Sigma$ and $\Delta$ as in Remark
\ref{rem-notacionSigmaDelta}.
%%%%%%%%%%%%%%%%%%%%%%%%
\begin{proposition}\cite[Theorem 5.1]{ReyesSuarez2018-3}\label{prop-indicsigma}
Let $A = \sigma(R)\langle x_1,\dotsc, x_n\rangle$ be a skew PBW
extension over $R$. Suppose that $\sigma_i\delta_j=\delta_j\sigma_i,
\delta_i\delta_j=\delta_j\delta_i$, and $\delta_k(d_{i,j}) =
\delta_k(r_l^{(i,j)}) = 0$, for $1\le i, j, k, l\le n$, where
$d_{i,j}$ and $r_l^{(i,j)}$ are the elements established in
Definition \ref{def.skewpbwextensions}. If
$\overline{\sigma_{k}}:A\to A$ and $\overline{\delta_k}:A\to A$ are
the functions given by
$\overline{\sigma_{k}}(f):=\sigma_k(a_0)+\sigma_k(a_1)X_1 + \dotsb +
\sigma_k(a_m)X_m$ and $\overline{\delta_k}(f):=\delta_k(a_0) +
\delta_k(a_1)X_1 + \dotsb + \delta_k(a_m)X_m$, for every $f=a_0 +
a_1X_1+\dotsb + a_mX_m\in A$, respectively, and
$\overline{\sigma_k}(r):=\sigma_k(r)$ and $\overline{\delta_k}(r) =
\delta_k(r)$, for every $1\le k\le n$ and $r\in R$, then
$\overline{\sigma_k}$ is an injective endomorphism of $A$ and
$\overline{\delta_k}$ is a $\overline{\sigma_k}$-derivation of $A$,
for each $k$.
\end{proposition}
%%%%%%%%%%%%%%%%%%%%%%%%
As we said in the Introduction, Krempa \cite{Krempa1996} introduced
the notion of rigidness as a generalization of reduced rings. The
second author in \cite[Definition 3.2]{Reyes2015} generalized the
notion of $\sigma$-rigidness to a finite family of endomorphisms of
a ring $R$. Let $R$ be a ring, $\Sigma =
\{\sigma_1,\dots,\sigma_n\}$ a family of endomorphisms of $R$ and
for $\alpha = (\alpha_1, \dotsc, \alpha_n) \in \mathbb{N}^n$,
$\sigma^{\alpha}:= \sigma_1^{\alpha_1}\circ \dotsc \circ
\sigma_n^{\alpha_n}$. $\Sigma$ is called a \emph{rigid endomorphisms
family of} $R$ if $a\sigma^{\alpha}(a)= 0$ implies $a = 0$, for
every $a\in R$ and $\alpha\in \mathbb{N}^n$. $R$ is called
$\Sigma$-\emph{rigid} if there exists a rigid endomorphisms family
$\Sigma$ of $R$. Now, if $\sigma$ is an endomorphism of $R$ and
$\delta$ is a $\sigma$-derivation of $R$, Annin \cite{Annin2004}
called $R$ a $\sigma$-{\em compatible} ring if for each $a, b\in R$,
$ab = 0$ if and only if $a\sigma(b)=0$ (necessarily, the
endomorphism $\sigma$ is injective). Notice that the
$\sigma$-compatible ring is a generalization of $\sigma$-rigid ring
to the more general case where $R$ is not assumed to be reduced (for
more details, see \cite{HashemiMoussavi2005}). $R$ is said to be
$\delta$-{\em compatible} if for each $a, b\in R$, $ab = 0$ implies
$a\delta(b)=0$. If $R$ is both $\sigma$-compatible and
$\delta$-compatible, then $R$ is called ($\sigma,\delta$)-{\em
compatible}. Notice that compatible rings are strictly more general
than rigid rings

\vspace{0.2cm}

Hashemi et al. \cite{HashemiKhalilAlhevaz2017} and the authors in
\cite {ReyesSuarez2018RUMA} and introduced independently the
$(\Sigma, \Delta)$-compatible rings which are a natural
generalization of $(\sigma, \delta)$-compatible rings. As expected
$(\Sigma, \Delta)$-compatible rings extend $\Sigma$-rigid rings.
Some examples and results of these compatible rings can be found in
\cite{Fajardoetal2020, HashemiKhalilAlhevaz2017, LouzariReyes2020,
ReyesSuarez2018RUMA, ReyesSuarez2019-2}.
%%%%%%%%%%%%%%%%%%%%%%%%
\begin{definition}
Consider a ring $R$ with a finite family of endomorphisms $\Sigma$
and a finite family of $\Sigma$-derivations $\Delta$.
\begin{enumerate}
\item [\rm (i)] \cite[Definition 3.1]{HashemiKhalilAlhevaz2017} or \cite[Definition 3.2]{ReyesSuarez2018RUMA} $R$ is said to be {\it $\Sigma$-compatible} if for each $a, b \in R$, $a\sigma^{\alpha}(b) = 0$ if and only if $ab = 0$, where $\alpha \in \mathbb{N}^n$. In this case, we say that $\Sigma$ is a {\it compatible family} of endomorphisms of $R$.
\item [\rm (ii)] \cite[Definition 3.1]{HashemiKhalilAlhevaz2017} or \cite[Definition 3.2]{ReyesSuarez2018RUMA} $R$ is said to be {\it $\Delta$-compatible} if for each $a, b \in R$, it follows that $ab = 0$ implies $a\delta^{\beta}(b)=0$, where $\beta \in \mathbb{N}^n$.
\item [\rm (iii)] \cite[Definition 3.1]{HashemiKhalilAlhevaz2017} or \cite[Definition 3.2]{ReyesSuarez2018RUMA} If $R$ is both $\Sigma$-compatible and $\Delta$-compatible, then $R$ is called {\it  $(\Sigma, \Delta)$-compatible}.
\item [\rm (iv)] \cite[Definition 4.1]{ReyesSuarez2019-2} $R$ is said to be {\it weak $\Sigma$-compatible} if for each $a, b \in R$, $a\sigma^{\alpha}(b)\in N(R)$ if and only if $ab \in N(R)$, where $\alpha \in \mathbb{N}^n$. $R$ is said to be {\it weak $\Delta$-compatible} if for each $a, b \in R$,  $ab \in N(R)$ implies $a\delta^{\beta}(b)\in N(R)$, where $\beta \in \mathbb{N}^n$. If $R$ is both weak $\Sigma$-compatible and weak $\Delta$-compatible, then $R$ is called {\it weak $(\Sigma, \Delta)$-compatible}.
\end{enumerate}
\end{definition}
%%%%%%%%%%%%%%%%%%%%%%%%
\begin{remark}\label{remark-comsiiSigma}
The composition of compatible endomorphisms is compatible. Indeed,
if $\sigma_1$ and $\sigma_2$ are compatible endomorphisms of $R$,
then for $a,b\in R$, $a\sigma_1\sigma_2(b)=
a(\sigma_1(\sigma_2(b))=0$ if and only if $a\sigma_2(b)=0$ if and
only if $ab=0$. By induction we obtain that the composition of a
finite family of compatible endomorphisms of a ring $R$ is
compatible. In this way, if $\Sigma=\{\sigma_1,\dots, \sigma_n\}$ is
a finite family of endomorphism of a ring $R$, then $R$ is
$\Sigma$-compatible if and only if $R$ is $\sigma_i$-compatible for
$1\leq i\leq n$. Moreover, $R$ is $\Sigma$-compatible if and only if
for $a,b\in R$, $\sigma^{\alpha}(a)b=0\Leftrightarrow ab=0$, for
each $\alpha\in \mathbb{N}^n$. Suppose that $\Sigma$ is a compatible
family of endomorphisms of $R$. By Remark \ref{remark-comsiiSigma}
we have that each $\sigma_k\in \Sigma$ is compatible and that
$\sigma^{\alpha}$ is compatible. By \cite[Lemma 2.1]{Alhevaz2012} we
have that $\sigma^{\alpha}(a)b=0\Leftrightarrow ab=0$. On the other
hand, if $\sigma^{\alpha}(a)b=0\Leftrightarrow ab=0$, for $\alpha\in
\mathbb{N}^n$, then  for $\alpha=(\alpha_1,\dots,\alpha_n)\in
\mathbb{N}^n$, where $\alpha_i=0$ if $i\neq k$ and $\alpha_k=1$, we
have that $\sigma^{\alpha}= \sigma_k\in \Sigma$, that is, $\sigma_k$
is compatible \cite[Lemma 2.1]{Alhevaz2012}. Remark
\ref{remark-comsiiSigma} shows that $\Sigma$ is compatible, and
therefore $R$ is $\Sigma$-compatible.
\end{remark}
%%%%%%%%%%%%%%%%%%%%%%%%
%%%%%%%%%%%%%%%%%%%%%%%%
\section{$\Sigma$-semicommutative rings}\label{sect-semicommt}
In this section, we define $\Sigma$-semicommutative rings as a
natural extension of $\sigma$-semicommutative rings introduced by
Ba\c{s}er et al. \cite{Baser}. We also consider some properties of
these rings which are used in later sections.
%%%%%%%%%%%%%%%%%%%%%%%%
\begin{definition}\label{def-sigma-semicomm}
Let $\Sigma=\{\sigma_1,\dots, \sigma_n\}$ be a finite family of
endomorphisms of a ring $R$. A ring $R$ is called
$\Sigma$-\emph{semicommutative} if whenever $ab = 0$ for $a, b\in
R$, $aR\sigma^{\alpha}(b) = 0$, for  every $\alpha\in
\mathbb{N}^n\setminus \{0\}$. A ring $R$ is called
\emph{$\Sigma$-semicommutative} if there exists a semicommutative
family of endomorphism $\Sigma$ of $R$.
\end{definition}
%%%%%%%%%%%%%%%%%%%%%%%%
It is clear that a semicommutative ring $R$ is
$\Sigma$-semicommutative, for $\Sigma=\{{\rm id}_R\}$, where ${\rm
id}_R$ is the identity endomorphism of $R$. Notice that if $R$ is
$\Sigma$-semicommutative and $S$ is a subring of $R$ with
$\sigma^{\alpha}(S)\subseteq S$, for each $\alpha\in
\mathbb{N}^{n}\setminus \{0\}$, then $S$ is also
$\Sigma$-semicommutative.

\vspace{0.2cm}

The following result was to be expected based on our motivation to
define the $\Sigma$-semicommutative rings.
%%%%%%%%%%%%%%%%%%%%%%%%
\begin{proposition}\label{prop-Sigmasemiciifeach}
Let $\Sigma=\{\sigma_1,\dots,\sigma_n\}$ be a finite family of
endomorphisms of a ring $R$. Then $R$ is $\Sigma$-semicommutative if
and only if $R$ is $\sigma_i$-semicommutative, for $1\leq i\leq n$.
\end{proposition}
%%%%%%%%%%%%%%%%%%%%%%%%
\begin{proof}
Suppose that $R$ is $\sigma_i$-semicommutative for each $i$. Notice
that composition of semicommutative endomorphisms is
semicommutative. Suppose that $\sigma_i$ and $\sigma_j$ are two
semicommutative endomorphisms of a ring $R$. Then $ab=0$, for
$a,b\in R$, implies $a\sigma_j(b)=0$, whence
$aR\sigma_i\sigma_j(b)=0$, and so $\sigma_i\sigma_j$ is
semicommutative. Since $R$ is $\sigma_i$-semicommutative, we get
that $R$ is $\sigma_i^2$-semicommutative. By induction it follows
that that $R$ is $\sigma_i^{\alpha_i}$-semicommutative for any
positive integer $\alpha_i$. Hence $R$ is
$\sigma^{\alpha}$-semicommutative for every
$\sigma^{\alpha}=\sigma_1^{\alpha_1}\dots \sigma_n^{\alpha_n}$,
where $(\alpha_1,\dotsc, \alpha_n)=\alpha\in \mathbb{N}^n\setminus
\{0\}$, i.e., $R$ is $\Sigma$-semicommutative.

For the  converse, if $R$ is $\Sigma$-semicommutative, then for
$a,b\in R$, $ab=0$ implies  $aR\sigma^{\alpha}(b)=0$, for every
$\alpha\in \mathbb{N}^n\setminus \{0\}$. In particular, when
$\sigma^{\alpha}=\sigma_i$, i.e., the $i$-th entry of $\alpha$ is 1
and zero otherwise, it follows that $aR\sigma_i(b)=0$, and therefore
$R$ is $\sigma_i$-semicommutative, for every $i$.
\end{proof}
%%%%%%%%%%%%%%%%%%%%%%%%
The following is an example of a $\Sigma$-semicommutative ring.
%%%%%%%%%%%%%%%%%%%%%%%%
\begin{example}\label{exampleS2}
Consider the ring
\begin{center}
    $R= \left \{ \begin{pmatrix}a & b\\ 0 & a \end{pmatrix} \mid a,b \in \mathbb{Z} \right \}$.
\end{center}
Let $\Sigma=\{\sigma_1,\sigma_2,\sigma_3\}$ be a family of
endomorphisms or $R$, where $\sigma_1 = {\rm id}_R$ is the identity
endomorphism of $R$, and $\sigma_2,\sigma_3$ defined as
\begin{center}
    $\sigma_2\left ( \begin{pmatrix}a & b\\ 0 & a \end{pmatrix} \right )= \begin{pmatrix}a & -b\\0 & a\end{pmatrix}, \ \ \ \ \ \ \sigma_3\left ( \begin{pmatrix}a & b\\ 0 & a \end{pmatrix} \right )= \begin{pmatrix}a & 0\\0 &
    a\end{pmatrix}$.
\end{center}
Notice that $\sigma_1$ and $\sigma_2$ are injective but $\sigma_3$
is not injective.

Let $AB = 0$ with
\begin{center}
    $A=\begin{pmatrix} a & a'\\ 0& a\end{pmatrix}\ \ \text{and}\ \ B=\begin{pmatrix} b & b'\\ 0 & b \end{pmatrix}.$
\end{center}
Then $ab = 0$ and $ab' + ba' = 0$. If $a=0$, then $a'b=0$, and so
$a'=0$ or $b=0$. If $b=0$, then $ab'=0$, and so $a=0$ or $b'=0$.
Now, let $U=\begin{pmatrix} r & r'\\ 0& r\end{pmatrix}$. In this
way,
\begin{center}
    $AU\sigma_1(B)=AUB=\begin{pmatrix} arb & arb' + ar'b + a'rb\\ 0& arb\end{pmatrix}=\begin{pmatrix} 0 & (ab' + a'b)r\\ 0& 0\end{pmatrix}=
    \begin{pmatrix} 0 & 0\\ 0& 0\end{pmatrix},$
\end{center}
since $ab = 0$ and $ab' + ba' = 0$. Therefore $R$ is
$\sigma_1$-semicommutative and so semicommutative. For $\sigma_2$,
\begin{align*}
    AU\sigma_2(B) = &\ \begin{pmatrix} a & a'\\ 0& a\end{pmatrix}\begin{pmatrix} r & r'\\ 0& r\end{pmatrix}\begin{pmatrix} b & -b'\\ 0 & b
    \end{pmatrix} \\
    = &\ \begin{pmatrix} arb & -arb' + ar'b - a'rb\\ 0& arb\end{pmatrix} \\
     = &\ \begin{pmatrix} 0 & -(ab' + a'b)r \\ 0 & 0\end{pmatrix} \\
     = &\ \begin{pmatrix} 0 & 0\\ 0& 0\end{pmatrix},
\end{align*}
which shows that $R$ is $\sigma_2$-semicommutative.

Finally,
\begin{center}
    $AU\sigma_3(B)=\begin{pmatrix} a & a'\\ 0& a\end{pmatrix}\begin{pmatrix} r & r'\\ 0& r\end{pmatrix}\begin{pmatrix} b & 0\\ 0 & b
    \end{pmatrix}= \begin{pmatrix} arb &  ar'b + a'rb\\ 0& arb\end{pmatrix}=\begin{pmatrix} 0 & a'rb\\ 0& 0\end{pmatrix}.$
\end{center}
If $a'rb\neq 0$, then $a'\neq 0$, $r\neq 0$ and $b\neq 0$. Since
$ab=0$, then $a=0$ and so $a'b=0$, because $ab' + ba' = 0$. As
$a'\neq 0$, then $b=0$, which is contradictory. Thus $a'rb= 0$ and
therefore $AU\sigma_3(B)=0$, which  implies that $R$ is
$\sigma_3$-semicommutative. From Proposition
\ref{prop-Sigmasemiciifeach} it follows that $R$ is a
$\Sigma$-semicommutative ring.
\end{example}
%%%%%%%%%%%%%%%%%%%%%%%%
The following is an example of a ring that is semicommutative but
not $\Sigma$-semicommutative.
%%%%%%%%%%%%%%%%%%%%%%%%
\begin{example}\label{ex-ex2Irano}
Consider the ring $R = \mathbb{K}[s, t]/\langle st\rangle$, where
$\mathbb{K}[s, t]$ is the commutative polynomial ring, and let
$\langle st\rangle$ be the ideal of $\mathbb{K}[s, t]$ generated by
$st$. Let $\overline{s}=s+\langle st\rangle$,
$\overline{t}=t+\langle st\rangle$ in $R$, and let $\sigma_1: R\to
R$ defined by $\sigma(\overline{s})= \overline{t}$,
$\sigma(\overline{t}) = \overline{s}$, and $\sigma_2: R \to R$ the
identity map. Notice that $R$ is semicommutative, since $R$ is
commutative. However, $R$ is not $\sigma_1$-semicommutative:
$\overline{s}\overline{t}=0$ but
$\overline{s}\sigma_1(\overline{t})=\overline{s}\overline{s}\neq 0$.
By Proposition \ref{prop-Sigmasemiciifeach}, $R$ is not
$\Sigma$-semicommutative for $\Sigma=\{\sigma_1,\sigma_2\}$.
\end{example}
%%%%%%%%%%%%%%%%%%%%%%%%
The following is an example of a ring that is
$\Sigma$-semicommutative but not semicommutative.
%%%%%%%%%%%%%%%%%%%%%%%%
\begin{example}\label{ex-Sigmasemicnosem} For the ring
\begin{center}
    $R= \left \{ \begin{pmatrix}a & b\\ 0 & c \end{pmatrix} \mid a,b, c \in \mathbb{Z} \right \}$,
\end{center}
and the endomorphism $\sigma$ of $R$ given by $\sigma\left
(\begin{pmatrix}a & b\\ 0 & c \end{pmatrix}\right )=\begin{pmatrix}a
& 0\\ 0 & 0
    \end{pmatrix}$, from \cite[Example 2.7]{Baser} we know that $R$ is $\Sigma$-semicommutative for $\Sigma=\{\sigma\}$. $R$ is not semicommutative, since for example $\begin{pmatrix}1 & 0\\ 0 & 0 \end{pmatrix}\begin{pmatrix}0 & 0\\ 0 & 1
    \end{pmatrix}=0$, but $\begin{pmatrix}1 & 0\\ 0 & 0 \end{pmatrix}\begin{pmatrix}1 & 1\\ 0 & 1 \end{pmatrix}\begin{pmatrix}0 & 0\\ 0 & 1 \end{pmatrix}= \begin{pmatrix}0 & 1\\ 0 & 0
    \end{pmatrix}\neq 0$.
\end{example}
%%%%%%%%%%%%%%%%%%%%%%%%
The following example shows that a ring $R$ is neither
semicommutative nor $\Sigma$-semicommutative.
%%%%%%%%%%%%%%%%%%%%%%%%
\begin{example}\label{ex-NoSigmasemicnosem} Let $R$ and $\sigma$ be as in Example \ref{ex-Sigmasemicnosem} and consider the endomorphism $\phi$ of $R$ given by $\phi\left (\begin{pmatrix}a & b\\ 0 & c \end{pmatrix}\right)=\begin{pmatrix}0 & 0\\ 0 &
    c \end{pmatrix}$. Notice that $R$ is $\Sigma$-semicommutative for $\Sigma=\{\sigma\}$, and $\begin{pmatrix}1 & 0\\ 0 & 0 \end{pmatrix}\begin{pmatrix}0 & 0\\ 0 & 1
    \end{pmatrix}=0$, but $\begin{pmatrix}1 & 0\\ 0 & 0 \end{pmatrix}\begin{pmatrix}1 & 1\\ 0 & 0 \end{pmatrix}\phi\left(\begin{pmatrix}0 & 0\\ 0 & 1
    \end{pmatrix}\right)= \begin{pmatrix}0 & 1\\ 0 & 0
    \end{pmatrix}\neq 0$, i.e., $R$ is not $\phi$-semicommutative. Proposition \ref{prop-Sigmasemiciifeach} shows that
    $R$ is not $\Sigma$-semicommutative, for $\Sigma=\{\sigma,
    \phi\}$. (recall that $R$ is not semicommutative by Example
    \ref{ex-Sigmasemicnosem}).
\end{example}
%%%%%%%%%%%%%%%%%%%%%%%%
Taking into account that for skew PBW extensions each of the
endomorphisms $\sigma_i\in \Sigma$ mentioned in Remark
\ref{rem-notacionSigmaDelta} is injective, from now on when we refer
to endomorphisms or families of ring endomorphisms, it will be
understood that these are injective endomorphisms, if not stated
otherwise.
%%%%%%%%%%%%%%%%%%%%%%%%
\begin{lemma}\label{lema-Sigmaidem}
If $\Sigma=\{\sigma_1,\dots,\sigma_n\}$ is a finite family of
endomorphisms of a $\Sigma$-semicommutative ring $R$, then we have
the following properties.
\begin{enumerate}
\item[\rm (1)] $\sigma^{\alpha}(1)=1$, for every $\alpha\in \mathbb{N}^n$.
\item[\rm (2)] $\sigma^{\alpha}(e)=e$, for each idempotent $e$ in $R$, and every $\alpha\in \mathbb{N}^n$.
\end{enumerate}
\end{lemma}
\begin{proof}
Suppose that $R$ is a $\Sigma$-semicommutative ring, with
$\Sigma=\{\sigma_1,\dots,\sigma_n\}$. Since every $\sigma_i$ is
assumed injective, it is clear that $\sigma^{\alpha}$ is injective,
for each $\alpha\in \mathbb{N}^n$.
\begin{enumerate}
\item[\rm (1)] If $\alpha=0=(0,\dots,0)$ then $\sigma^{\alpha}$ is the identity endomorphism of $R$, thus $\sigma^{\alpha}(1)=1$. Suppose that $\alpha\in \mathbb{N}^n\setminus \{0\}$. Notice that $0 = \sigma^{\alpha}(1)(1-\sigma^{\alpha}(1))=\sigma^{\alpha}(1)-\sigma^{\alpha}(\sigma^{\alpha}(1))$. Since $R$ is $\Sigma$-semicommutative, $0 = \sigma^{\alpha}(\sigma^{\alpha}(1))$ implies $0=\sigma^{\alpha}(1)\sigma^{\alpha}(1-\sigma^{\alpha}(1))=\sigma^{\alpha}(1)\sigma^{\alpha}(1)-\sigma^{\alpha}(1)\sigma^{\alpha}(\sigma^{\alpha}(1)) = \sigma^{\alpha}(1)-\sigma^{\alpha}(\sigma^{\alpha}(1))$. Then $\sigma^{\alpha}(1)= \sigma^{\alpha}(\sigma^{\alpha}(1))$. Therefore $ 1= \sigma^{\alpha}(1)$, since  $\sigma^{\alpha}$ is injective.
\item[\rm (2)] If $e$ is an idempotent of $R$, then $e(1-e)=0$ implies $0 = e\sigma^{\alpha}(1-e)=e\sigma^{\alpha}(1)-e\sigma^{\alpha}(e)=e-e\sigma^{\alpha}(e)$, since $R$ is $\Sigma$-semicommutative and by part (1). Thus $e = e\sigma^{\alpha}(e)$. Now, $(1-e)e=0$ implies $0=(1-e)\sigma^{\alpha}(e)=\sigma^{\alpha}(e)-e\sigma^{\alpha}(e)$, since $R$ is $\Sigma$-semicommutative. Hence, $\sigma^{\alpha}(e)=e\sigma^{\alpha}(e)$, and so $\sigma^{\alpha}(e)=e$.
\end{enumerate}
\end{proof}
%%%%%%%%%%%%%%%%%%%%%%%%
\begin{corollary}\cite[Lemma 2.1]{Kim2014}
If $R$ is a $\sigma$-semicommutative ring, then the following
statements hold:
\begin{enumerate}
\item[\rm (1)] $\sigma(1)= \sigma^k(1)$, for all $k\geq 2$.
\item[\rm (2)] $\sigma(e)=e\sigma(e)=e\sigma(1)=e\sigma^k(e)$, for all $k\geq 2$ and every idempotent $e$ of $R$.
\item[\rm (3)] If $\sigma$ is a monomorphism, then $\sigma(1)=1$, and hence $\sigma(e)=e$, for all idempotent element $e\in R$.
\item[\rm (4)] If $\sigma$ is an epimorphism, then $\sigma(e)=e$, for every idempotent $e$ of $R$.
\end{enumerate}
\end{corollary}
%%%%%%%%%%%%%%%%%%%%%%%%
\begin{proposition}\label{prop-sigmaimplAbel}
$\Sigma$-semicommutative rings are Abelian.
\end{proposition}
\begin{proof}
Let $R$ be a $\Sigma$-semicommutative ring and let $e$ be an
idempotent in $R$. By the $\Sigma$-semicommutative property of $R$,
$e(1-e)=0=(1-e)e$ implies
$eR\sigma^{\alpha}(1-e)=0=(1-e)R\sigma^{\alpha}(e)$, for every
${\alpha}\in \mathbb{N}\setminus \{0\}$. From Lemma
\ref{lema-Sigmaidem} (2) we know that $eR(1-e)=0=(1-e)Re$, i. e.,
$er(1-e)=er-ere=0=(1-e)re=re-ere$, for all $r\in R$. In this way,
$er=re$ for all $r\in R$, and therefore $R$ is Abelian.
\end{proof}
%%%%%%%%%%%%%%%%%%%%%%%%
\begin{corollary} \cite[Proposition
2.6]{Baser} \label{cor-genprop2.6} Let $R$ be a
$\sigma$-semicommutative ring.  If $\sigma(1) = 1$, then $R$ is
Abelian.
\end{corollary}
%%%%%%%%%%%%%%%%%%%%%%%%
The ring $R$ in Example \ref{ex-Sigmasemicnosem} is
$\sigma$-semicommutative, but $\sigma(1)\neq 1$ and $R$ is not
Abelian \cite[Example 2.7]{Baser}.

\vspace{0.2cm}

The following theorem generalizes \cite[Theorem 2.4]{Baser}.
%%%%%%%%%%%%%%%%%%%%%%%%
\begin{theorem}\label{teo-Sigmarigidiffreduc}
A ring $R$ is $\Sigma$-semicommutative and reduced if and only if
$R$ is $\Sigma$-rigid.
\end{theorem}
\begin{proof}
Suppose that $R$ is a $\Sigma$-rigid ring, for
$\Sigma=\{\sigma_1,\dots,\sigma_n\}$ a finite family of
monomorphisms of $R$. If $r^2=0$ for $r\in R$ then $0 =
\sigma^{\alpha}(r^2)=r\sigma^{\alpha}(r^2)\sigma^{\alpha}(\sigma^{\alpha}(r))=r\sigma^{\alpha}(r)\sigma^{\alpha}(r)\sigma^{\alpha}(\sigma^{\alpha}(r))
= r\sigma^{\alpha}(r)\sigma^{\alpha}(r\sigma^{\alpha}(r))$, for
every $\alpha\in \mathbb{N}^n$. Then $r\sigma^{\alpha}(r)=0$, since
$R$ is $\Sigma$-rigid, whence $r=0$, that is, $R$ is reduced. Now,
assume that $ab=0$, for $a,b\in R$. Then $0=baba=(ba)^2$ implies
$ba=0$, since $R$ is reduced, whence
$0=ar\sigma^{\alpha}(ba)\sigma^{\alpha}(r)\sigma^{\alpha}(\sigma^{\alpha}(b)=
ar\sigma^{\alpha}(b)\sigma^{\alpha}(a)\sigma^{\alpha}(r)\sigma^{\alpha}(\sigma^{\alpha}(b)
= ar\sigma^{\alpha}(b)\sigma^{\alpha}(ar\sigma^{\alpha}(b))$. Thus,
$ar\sigma^{\alpha}(b)=0$, for all $r\in R$, since $R$ is
$\Sigma$-rigid. Then $aR\sigma^{\alpha}(b)=0$, which implies that
$R$ is $\Sigma$-semicommutative.

For the converse, suppose that $R$ is reduced and
$\Sigma$-semicommutative. If $r\sigma^{\alpha}(r)=0$ for all $r\in
R$, then $\sigma^{\alpha}(r)r=0$, since $R$ is reduced. So, by the
$\Sigma$-semicommutative property of $R$,
$\sigma^{\alpha}(r)R\sigma^{\alpha}(r)=0$, whence
$0=\sigma^{\alpha}(r)\sigma^{\alpha}(r)=\sigma^{\alpha}(r^2)$, which
implies that $r^2 = 0$, since $\sigma^{\alpha}$ is a monomorphism.
By assumption $R$ is reduced, so $r = 0$, and hence $R$ is
$\Sigma$-rigid.
\end{proof}
%%%%%%%%%%%%%%%%%%%%%%%%
The following example shows that the reducibility condition and the
injectivity condition in Theorem \ref{teo-Sigmarigidiffreduc} cannot
be omitted.
%%%%%%%%%%%%%%%%%%%%%%%%
\begin{example}\label{exam-notred}
Let $R$ the $\Sigma$-semicommutative ring as in Example
\ref{exampleS2}. Notice that $R$ is not reduded, since for example
$\begin{pmatrix}0 & 1\\ 0 & 0
\end{pmatrix}^2=\begin{pmatrix}0 & 0\\ 0 & 0 \end{pmatrix}$. Also, $R$ is not $\Sigma$-rigid:  $\begin{pmatrix}0 & 1\\ 0 & 0
\end{pmatrix}\sigma_2\sigma_3\left (\begin{pmatrix}0 & 1\\ 0 & 0
\end{pmatrix}\right)=\begin{pmatrix}0 & 0\\ 0 & 0
\end{pmatrix}$.
\end{example}
%%%%%%%%%%%%%%%%%%%%%%%%
\begin{corollary}\cite[Theorem 2.4]{Baser}\label{cor-teo2.4} A ring $R$ is $\sigma$-rigid if and only if $R$ is a $\sigma$-semicommutative reduced
 ring and $\sigma$ is a monomorphism.
\end{corollary}
%%%%%%%%%%%%%%%%%%%%%%%%
Mohammadi et al. in \cite{Mohammadi2016} defined a new class of
rings called nil-reversible rings. A ring $R$ is called
\emph{nil-reversible} if for every $a\in R$, $b\in N(R)$, $ab = 0$
if and only if $ba = 0$. Nil-reversible rings are a generalization
of reversible rings. On the other hand, the authors
\cite{SuarezHigueraReyes2021} defined and studied rings with
$\Sigma$-skew reflexive nilpotent property ($\Sigma$-skew RNP
rings). Let $\Sigma=\{\sigma_1,\dots, \sigma_n\}$ be a finite family
of endomorphisms of a ring $R$. $\Sigma$ is called a \emph{right}
(resp., \emph{left}) {\em skew RNP family} if for $a, b\in N(R)$,
$aRb = 0$ implies that $bR\sigma^{\alpha}(a) = 0$ (resp.
$\sigma^{\alpha}(b)Ra = 0$), for  every $\alpha\in \mathbb{N}^n$.
$R$ is said to be right (resp., left) $\Sigma$-{\em skew RNP} if
there exists a right (resp., left) skew RNP family of endomorphisms
$\Sigma$ of $R$. If $R$ is both right and left $\Sigma$-skew RNP,
then we say that $R$ is $\Sigma$-{\em skew RNP} \cite[Definition
4.1]{SuarezHigueraReyes2021}. For nil-reversible rings the
$\Sigma$-semicommutative and $\Sigma$-skew RNP properties are
equivalent.
%%%%%%%%%%%%%%%%%%%%%%%%
\begin{theorem}\label{teo-SimasemiciffSigmaskewRNP}
If $R$ is a nil-reversible ring, then $R$ is
$\Sigma$-semicommutative if and only if $R$ is right $\Sigma$-skew
RNP.
\end{theorem}
\begin{proof}
Suppose that $R$ is $\Sigma$-semicommutative and let $a,b\in N(R)$
such that $aRb=0$. Then $ab=0$, and by the nil-reversibility of $R$,
$ba=0$, whence $bR\sigma^{\alpha}(a)=0$, since $R$ is
$\Sigma$-semicommutative. This shows that $R$ is right $\Sigma$-skew
RNP.

For the converse, if $ab=0$ with $a,b\in R$, then $Rab=0$ and thus
$bRa=0$ by the nil-reversibility of $R$. Thus $aR\sigma^{\alpha}(a)
= 0$, since $R$ is right $\Sigma$-skew RNP, and so $R$ is
$\Sigma$-semicommutative.
\end{proof}
%%%%%%%%%%%%%%%%%%%%%%%%
%%%%%%%%%%%%%%%%%%%%%%%%
\section{Skew PBW extensions over $\Sigma$-semicommutative rings}\label{sect-sPBWextoverSigma}
Let $A=\sigma(R)\langle x_1, \dots, x_n\rangle$ be a skew PBW
extension over a ring $R$. We start this section with the following
property that extends \cite[Theorem 3.3 (2)]{Baser}.
%%%%%%%%%%%%%%%%%%%%%%%%
\begin{proposition}\label{prop-AsemimplRSigmasem}
If $A$ is a semicommutative skew PBW extension over a ring $R$, then
$R$ is $\Sigma$-semicommutative.
\end{proposition}
\begin{proof}
Let $a,b\in R$ such that $ab=0$. By the semicommutativity of $A$ we
have $aRx_ib=0$, for $1\leq i\leq n$. Proposition
\ref{sigmadefinition} implies that $ar\sigma_i(b)x_i+\delta_i(b)=0$,
for all $r\in R$, whence $aR\sigma_i(b)=0$. This shows that $R$ is
$\sigma_i$-semicommutative, for $1\leq i\leq n$, and by Proposition
\ref{prop-Sigmasemiciifeach}, $R$ is $\Sigma$-semicommutative.
\end{proof}
%%%%%%%%%%%%%%%%%%%%%%%%
The following example presents a skew PBW extension over a
semiconmutative ring that is not  semicommutative.
%%%%%%%%%%%%%%%%%%%%%%%%
\begin{example}\label{ex-nosemicom}
Let $R=\mathbb{K}[s,t]/\langle st\rangle$ and $\Sigma=\{\sigma_1,
\sigma_2\}$ as in Example \ref{ex-ex2Irano}. Let $\theta_2: R[x_1;
\sigma_1]\to R[x_1; \sigma_1]$ the identity map. Notice that
$\sigma_1$ and $\sigma_2=\theta|_R$ are automorphisms of $R$, whence
$A=R[x_1; \sigma_1][x_2; \theta_2]$ is an iterated skew polynomial
ring of injective type \cite[Example 1.1.5 (iii)]{Fajardoetal2020},
and therefore $A$ is a skew PBW extension over $R=\mathbb{K}[s,
t]/\langle st\rangle$, i.e., $A=\sigma(R)\langle x_1,x_2\rangle$.
According to Proposition \ref{prop-AsemimplRSigmasem}, $A$ is not
semicommutative since $R$ is not $\Sigma$-semicommutative. For
example $\overline{s}\overline{t}=0$ but
$\overline{s}x_1\overline{t}=\overline{s}
\sigma_1(\overline{t})x_1=\overline{s}^2x_1\neq 0$.
\end{example}
%%%%%%%%%%%%%%%%%%%%%%%%
According to Rege and Chhawchharia \cite{RegeChhawchharia1997}, a
ring $R$ is called {\em Armendariz} if for elements $f(x) =
\sum_{i=0}^{m} a_ix^{i},\ g(x) = \sum_{j=0}^{l} b_jx^{j}\in R[x]$,
$f(x)g(x) = 0$ implies $a_ib_j = 0$, for all $i, j$. Hirano
\cite{Hirano2002} generalized this notion to quasi-Armendariz rings.
A ring $R$ is called {\em quasi-Armendariz} if for $f(x) =
\sum_{i=0}^{m} a_ix^{i},\ g(x) = \sum_{j=0}^{l} b_jx^{j}$ in $R[x]$,
$f(x)R[x]g(x) = 0$ implies $a_iRb_j = 0$, for all $i, j$. It is well
known that reduced rings are Armendariz and Armendariz rings are
quasi-Armendariz, but the converse are not true in general
\cite[Proposition 2.1]{Rege1997}. Zhao et al.  \cite{Zhao2013}
proved that for a quasi-Armendariz ring $R$, $R$ is reflexive if and
only if $R[x]$ is reflexive if and only if $R[x; x^{-1}]$ is
reflexive. Hashemi et al. \cite{HashemiMoussavi2005} investigated
also a generalization of $\sigma$-rigid rings by introducing the
condition $(SQA1)$ which is a version of quasi-Armendariz rings for
Ore extensions. Let $R[x;\sigma,\delta]$ be a Ore extension of a
ring $R$. We say that $R$ satisfies the $(SQA1)$ condition, if
whenever $f(x)R[x; \sigma, \delta]g(x) = 0$ for $f(x) = \sum_{i=0}^m
a_ix^i$ and $g(x) = \sum_{j=0}^m b_jx^j \in R[x; \sigma, \delta]$,
then $a_iRb_j = 0$, for all $i, j$. Notice that if $R$ is
$\sigma$-rigid, then $R$ satisfies also (SQA1).

\vspace{0.2cm}

In \cite{ReyesSuarez2018RUMA}, the notions $(SA1)$ and $(SQA1)$ are
defined for skew PBW extensions with the aim of generalizing the
results about Baer, quasi-Baer, $p.p.$ and $p.q.$-Baer rings for
$\Sigma$-rigid rings, to the $(\Sigma,\Delta)$-compatible rings. Let
$A$ be a skew PBW extension of $R$. Following \cite[Definition
4.1]{ReyesSuarez2018RUMA}, we say that $R$ satisfies the condition
({\em SA1}) if whenever $fg = 0$ for $f=a_0+a_1X_1+\dotsb + a_mX_m$
and $g=b_0 + b_1Y_1 + \dotsb + b_tY_t$ elements of $A$, then $a_ib_j
= 0$, for every $i, j$. Now, from \cite[Definition
4.11]{ReyesSuarez2018RUMA}, we say that $R$ satisfies the ({\em
SQA1}) condition if whenever $fAg=0$ for $f=a_0+a_1X_1+\dotsb +
a_mX_m$ and $g=b_0 + b_1Y_1 + \dotsb + b_tY_t$ elements of $A$, then
$a_iRb_j = 0$, for every $i, j$. By \cite[Definition
6.1.2]{Fajardoetal2020}, we say that $R$ is a $(\Sigma,
\Delta)$\emph{-Armendariz} ring if for polynomials $f = a_0 + a_1X_1
+ \dotsb + a_mX_m$ and $g= b_0 + b_1Y_1 + \dotsb + b_tY_t$ in $A$,
the equality $fg=0$ implies $a_iX_ib_jY_j=0$, for every $i, j$. We
say that $R$ is an $(\Sigma, \Delta)$\emph{-weak Armendariz }ring,
if for linear polynomials $f = a_0 + a_1x_1 + \dotsb + a_nx_n$ and
$g= b_0 + b_1x_1 + \dotsb + b_nx_n$ in $A$, the equality $fg=0$
implies $a_ix_ib_jx_j=0$, for every $i, j$. $R$ is called a
$\Sigma$\emph{-skew Armendariz} ring, if for elements
$f=\sum_{i=0}^{m}a_iX_i$ and $g = \sum_{j=0}^{t} b_jY_j$ in $A$, the
equality $fg=0$ implies $a_i\sigma^{\alpha_i}(b_j)=0$,  for all
$0\le i\le m$ and $0\le j\le t$, where $\alpha_i={\rm exp}(X_i)$.
$R$ is called a \emph{weak $\Sigma$-skew Armendariz} ring, if for
elements $f=\sum_{i=0}^{n} a_ix_i$ and $g = \sum_{j=0}^{n} b_jx_j$
in $A$ $(x_0:=1)$, the equality $fg=0$ implies $a_i\sigma_i(b_j)=0$,
for all $0\le i, j \le n$ $(\sigma_0:={\rm id}_R)$ \cite[Definition
6.1.3]{Fajardoetal2020}. In \cite{ReyesSuarez2018-3} the authors
defined skew-Armendariz rings and studied some properties of skew
PBW extensions over skew-Armendariz rings. We say that $R$ is a
\emph{skew-Armendariz} ring, if for polynomials
$f=\sum_{i=0}^{m}a_iX_i$ and $g = \sum_{j=0}^{t} b_jY_j$ in $A$, $fg
= 0$ implies $a_0b_k = 0$, for each $0\leq k \leq t$
\cite[Definition 4.1]{ReyesSuarez2018-3}.

\vspace{0.2cm}

From the definitions above we can establish the following relations
\cite[p. 3204]{ReyesSuarez2018-3}:

{\small{\begin{equation*} \Sigma{\rm-rigid}\ \subsetneqq\
        (\Sigma, \Delta){\rm -Armendariz}\ \subsetneqq\ \Sigma{\rm-skew\
            Armendariz} \ \subsetneqq\ {\rm skew-Armendariz},
        \end{equation*}}}
{\small {\begin{equation*} \Sigma{\rm-rigid}\ \subsetneqq\ (\Sigma,
        \Delta){\rm -weak\ Armendariz}\ \subsetneqq\ {\rm weak}\
        \Sigma{\rm-skew\ Armendariz}.
        \end{equation*}}}
%%%%%%%%%%%%%%%%%%%%%%%%
The following theorem relates $\Sigma$-semicommutative rings to
another special class of rings such as reduced, $\Sigma$-rigid,
$(\Sigma,\Delta)$-Armendariz, $(\Sigma,\Delta)$-compatible and
skew-Armendariz which have been studied in \cite{Fajardoetal2020,
HashemiKhalilAlhevaz2017, ReyesSuarez2018-3, ReyesSuarez2018RUMA,
ReyesSuarez2019-1, ReyesSuarez2019-2, ReyesSuarez2019Radicals,
ReyesSuarez2021Discr}.
%%%%%%%%%%%%%%%%%%%%%%%%
\begin{theorem}\label{teo-compilSigma-semic}
If $A=\sigma(R)\langle x_1, \dots, x_n\rangle$ is a skew PBW
extension over $R$, then the following statements are equivalent:
\begin{enumerate}
\item[\rm (1)] $R$ is $\Sigma$-semicommutative and reduced.
\item[\rm (2)] $R$ is $\Sigma$-rigid.
\item[\rm (3)] $R$ is $(\Sigma,\Delta)$-Armendariz and reduced.
\item[\rm (4)] $A$ is reduced.
\item[\rm (5)] $R$ is reduced and $(\Sigma,\Delta)$-compatible.
\item[\rm (6)] $R$ is reduced and skew-Armendariz.
\end{enumerate}
\end{theorem}
\begin{proof}
(1) $\Leftrightarrow$ (2) It follows from Theorem
\ref{teo-Sigmarigidiffreduc}.

(2) $\Leftrightarrow$ (3) $\Leftrightarrow$ (4) The assertion is a
direct consequence of \cite[Corollary 6.1.10]{Fajardoetal2020}.

(4) $\Leftrightarrow$ (5) It follows from \cite[Theorem
3.9]{ReyesSuarez2018RUMA}.

(6) $\Leftrightarrow$ (2) It is a corollary from \cite[Theorem
4.4]{ReyesSuarez2018-3}.
\end{proof}
%%%%%%%%%%%%%%%%%%%%%%%%
\begin{lemma}\label{lema-propidempot}
If $A$ is a skew PBW extension over a $\Sigma$-semicommutative ring
$R$ and $e$ is an idempotent of $R$, then the following assertions
hold:
\begin{enumerate}
\item[\rm (1)] $\delta^{\beta}(e)=0$, for every $\beta\in \mathbb{N}\setminus \{0\}$.
\item[\rm (2)] $x^{\alpha}e=ex^{\alpha}$, for all $x^{\alpha}\in {\rm Mon}(A)$.
\item[\rm (3)] $fe=ef$, for all $f\in A$, i.e., $Ae=eA$ and thus $eA$ is an ideal of $A$.
\end{enumerate}
\end{lemma}
\begin{proof}
\begin{enumerate}
\item [\rm (1)] Let $A=\sigma(R)\langle x_1, \dots, x_n\rangle$ be a skew PBW extension and $(\Sigma, \Delta)$ the system of endomorphisms and $\Sigma$-derivations of $R$ with respect to $A$, as in Remark \ref{rem-notacionSigmaDelta}. Suppose  that $R$ is $\Sigma$-semicommutative and let $e$ be an idempotent of $R$.
If $\delta_i\in \Delta$, then
$\delta_i(e)=\delta_i(ee)=\sigma_i(e)\delta_i(e) +\delta_i(e)e$,
since $\delta_i$ is a $\sigma_i$-derivation. By Lemma
\ref{lema-Sigmaidem} we have that $\sigma_i(e)=e$ and from
Proposition \ref{prop-sigmaimplAbel} we have that
$\delta_i(e)e=e\delta_i(e)$. Then $\delta_i(e)=2e\delta_i(e)$ and
thus $\delta_i(e)=0$. Since $\delta_i$ is arbitrary, then
$\delta^{\beta}(e)=\delta_1^{\beta_1}\dots \delta_n^{\beta_n}(e)=0$,
for every $(\beta_1,\dots,\beta_n)=\beta\in \mathbb{N}\setminus
\{0\}$.
\item [\rm (2)] Let $x^{\alpha}=x_1^{\alpha_1}x_2^{\alpha_2}\dotsb x_{n-1}^{\alpha_{n-1}}x_n^{\alpha_n}\in {\rm Mon}(A)$. Since $\sigma^{\alpha}(e)=e$ for each $\alpha\in \mathbb{N}^n\setminus \{0\}$ (for Lemma \ref{lema-Sigmaidem} (2)) and $\delta_i(e)=0$, for $1\leq i\leq n$ (due to (1)), then $\delta_i(\sigma^{\gamma}(e))=\delta_i(e)=0$, for each $\gamma\in \mathbb{N}^n\setminus \{0\}$ and $1\leq i\leq n$. Thus, by Remark \ref{remark-remark2.8} we have that
$x^{\alpha}(e)=x_1^{\alpha_1}x_2^{\alpha_2}\dotsb
x_{n-1}^{\alpha_{n-1}}x_n^{\alpha_n}e=
(\sigma_n^{\alpha_n}(e))))x_1^{\alpha_1}\dotsb
x_n^{\alpha_n}=\sigma^{\alpha}(e)x^{\alpha}= ex^{\alpha}$.
\item [\rm (3)] Let $f=\sum_{k=0}^ma_kX_k\in A$, with $X_k\in {\rm Mon}(A)$ and $a_k\in R$, for $1\leq k\leq m$. Then $fe=\sum_{k=0}^m(a_kX_ke)$. By (2) above we have $X_ke=eX_k$, for $1\leq k\leq m$. Thus, $a_kX_ke=a_keX_k$, for $1\leq k\leq m$, and by Proposition \ref{prop-sigmaimplAbel} we know that $R$ is Abelian, so $a_keX_k=ea_kX_k$ for $1\leq k\leq m$. Therefore $fe=ef$.
\end{enumerate}
\end{proof}
%%%%%%%%%%%%%%%%%%%%%%%%
\begin{proposition}\label{prop-anulad}
If $R$ is a $\Sigma$-semicommutative Baer ring $R$ for
$\Sigma=\{\sigma_1,\dots,\sigma_n\}$ a family of endomorphisms of
$R$, and $\Delta=\{\delta_i,\dots,\delta_n\}$ is a family of
$\sigma_i$-derivations for each $\delta_i$, then for all $\alpha,
\beta\in \mathbb{N}^n$ and every non-empty subset $S$ of $R$, the
following assertions hold:
\begin{enumerate}
\item[\rm (1)] $l_R(S)=l_R(\sigma^{\alpha}(S))$.
\item[\rm (2)] $l_R(S)\subseteq l_R(\delta^{\beta}(S))$.
\item[\rm (3)] $\sigma_i(l_R(S))\subseteq l_R(S)$.
\item[\rm (4)] $\delta_i(l_R(S))\subseteq l_R(S)$.
\end{enumerate}
\end{proposition}
\begin{proof}
By assumption $R$ is Baer, so there exist an idempotent element $e$
in $R$ such that $l_R(S)=Re$.
\begin{enumerate}
\item[\rm (1)] By the Baer property of $R$, for the nonempty subset $\sigma_1^{\alpha_1}\cdots\sigma_n^{\alpha_n}(S)=\sigma^{\alpha}(S)$, there exist an idempotent element $e'$ in $R$ such that $l_R(\sigma^{\alpha}(S))=Re'$. Since $eS=0$ and $R$ is $\Sigma$-semicommutative, then $eR\sigma^{\alpha}(S)=0$ for $\alpha\in \mathbb{N}^n\setminus \{0\}$, and so $e\sigma^{\alpha}(S)=0$. Thus $e\in l_R(\sigma^{\alpha}(S))=Re'$ and $e=ee'=e'e$. By Lemma \ref{lema-Sigmaidem} (2) we have that $\sigma_1^{\alpha_1}\cdots\sigma_n^{\alpha_n}(e')=\sigma^{\alpha}(e')=e'$, thus $0=e'\sigma^{\alpha}(S)=\sigma^{\alpha}(e')\sigma^{\alpha}(S)=\sigma^{\alpha}(e'S)$ implies that $e'S=0$, since $\sigma^{\alpha}$ is a monomorphism of $R$. Then $e'\in l_R(S)$ and $e'=e'e=ee'$. Therefore $e=e'$ and so $l_R(S)=Re=Re'=l_R(\sigma^{\alpha}(S))$.
\item [\rm (2)] Let $a\in l_R(S)=Re$, then $a=re$ for some $r\in R$. By Proposition \ref{prop-sigmaimplAbel} we have that $R$ is Abelian, so $a=er=re$. Since $e\in l_R(S)$ then  $es=0$, for all $s\in S$. Thus $0 = \delta_2(es)=\sigma_2(e)\delta_2(s)+\delta_2(e)s=e\delta_2(s)$, since $\sigma_2(e)=e$ (by Lemma \ref{lema-Sigmaidem} (2)) and $\delta_2(e)=0$ (by Lemma \ref{lema-propidempot} (1)). Thus $ 0 = \delta_1(e\delta_2(s))=\sigma_1(e)\delta_1(\delta_2(s))+\delta_1(e)\delta_2(s)$. Again, Lemma \ref{lema-Sigmaidem} (2) and Lemma \ref{lema-propidempot} (1) imply that $e\delta_1\delta_2(s)=0$, for all $s\in S$. Continuing with this procedure, even repeating subindexes, we have that $e\delta_1^{\beta_1}\cdots \delta_n^{\beta_n}(S)=e\delta^{\beta}(S)=0$ for all $\beta\in \mathbb{N}^n\setminus \{0\}$. Then $0 = re\delta^{\beta}(S)=a\delta^{\beta}(S)$, and therefore $a\in l_R(\delta^{\beta}(S)$ for all $\beta\in \mathbb{N}^n\setminus \{0\}$.
\item [\rm (3)] The fact $r\in l_R(S)$ (thus $\sigma_i(r)\in \sigma_i(l_R(S)$) implies that $rs=0$, for all $s\in S$, and so $\sigma_i(rs)=\sigma_i(r)\sigma_i(s)=0$, whence $\sigma_i(r)\in l_R(\sigma_i(S))=l_R(S)$, by (1).
\item [\rm (4)] Since $\delta_i$ is a $\sigma_i$ derivation, we have that $\delta_i(rs)=\sigma_i(r)\delta_i(s)+\delta_i(r)s$. Thus, $rs=0$ implies $\sigma_i(r)\delta_i(s)=-\delta_i(r)s$. Therefore, $r\in l_R(S)$ implies that  $rs=0$ for all $s\in S$, and so $\sigma_i(r)\delta_i(s)=-\delta_i(r)s$. Now, from (2), $rs=0$ implies  $r\delta_i(s)=0$, and thus $r\in l_R(\delta_i(S))$, and from (3), $\sigma_i(r)\in l_R(\delta_i(S))$. In this way, $\sigma_i(r)\delta_i(s)=0$, for all $s\in S$, and so  $\delta_i(r)s=0$, i. e., $\delta_i(r)\in l_R(S)$.
\end{enumerate}
\end{proof}
%%%%%%%%%%%%%%%%%%%%%%%%
\begin{proposition}\label{prop-SemBaerimplskewArm}
If $A$ is a skew PBW extension over a $\Sigma$-semicommutative Baer
ring $R$, then $R$ is reduced.
\end{proposition}
\begin{proof}
By Proposition \ref{prop-sigmaimplAbel} we know that $R$ is Abelian,
and so reduced (recall that Abelian Baer rings are reduced
\cite{Birkenmeier2013}).
\end{proof}
%%%%%%%%%%%%%%%%%%%%%%%%
\begin{corollary}\label{cor-relatrigid}
If $A$ is a skew PBW extension over a $\Sigma$-semicommutative Baer
ring $R$, then the following statements hold:
\begin{enumerate}
\item[\rm (1)] $R$ is $\Sigma$-rigid.
\item[\rm (2)] $A$ is reduced, and therefore $A$ is abelian.
\item[\rm (3)] $R$ is $\Sigma$-skew Armendariz.
\item[\rm (4)] $R$ is $(\Sigma, \Delta)$-compatible.
\item [\rm (5)] $R$ is weak symmetric if and only if $A$ is weak symmetric.
\item [\rm (6)] Both $R$ and $A$ are reflexive.
\item [\rm (7)] $R$ is $\Sigma$-skew RNP.
\end{enumerate}
\end{corollary}
%%%%%%%%%%%%%%%%%%%%%%%%
\begin{proof}
\begin{enumerate}
\item[\rm (1)] From Proposition \ref{prop-SemBaerimplskewArm} we have that $R$ is reduced. Thus, by Theorem \ref{teo-compilSigma-semic} we have that $R$ is $\Sigma$-rigid.
\item[\rm (2)] $A$ is reduced by (1) and Theorem \ref{teo-compilSigma-semic}. Now, reduced rings are easily shown to be abelian.
\item[\rm (3)] It follows from (2) and \cite[Corollary 6.1.10]{Fajardoetal2020}.
\item[\rm (4)] Since $R$ is $\Sigma$-rigid, the assertion is a consequence of \cite[Theorem 3.9]{ReyesSuarez2018RUMA}.
\item[\rm (5)] It follows from (1) and \cite[Corollary 1]{ReyesSuarez2021Discr}.
\item[\rm (6)] The assertion is a consequence of (1) and \cite[Proposition 4.12]{SuarezHigueraReyes2021}.
\item[\rm (7)] Again, the assertion follows from item (1) above and \cite[Proposition 4.12]{SuarezHigueraReyes2021}.
\end{enumerate}
\end{proof}
%%%%%%%%%%%%%%%%%%%%%%%%
\begin{theorem}\label{teo-equivBaer}
If $A = \sigma(R)\langle x_1,\dotsc, x_n\rangle$ is a skew PBW
extension over a $\Sigma$-semicommutative and $\Delta$-compatible
ring $R$, then the following statements are equivalent:
\begin{enumerate}
\item[\rm (1)] $R$ is Baer.
\item[\rm (2)] $R$ is quasi-Baer.
\item[\rm (3)] $A$ is quasi-Baer.
\item[\rm (4)] $A$ is Baer.
\end{enumerate}
\end{theorem}
\begin{proof}
(2) $\Rightarrow$ (1) Suppose that $R$ is a quasi-Baer ring and let
$S$ be a nonempty subset of $R$, then $r_R(SR)=eR$ for some
idempotent $e$ of $R$. Thus $r_R(SR)=eR\subseteq r_R(S)$, since
$S\subseteq SR$. Now,  $a\in r_R(S)$ implies $Sa=0$ and this implies
that $SR\sigma^{\alpha}(a)=0$ for $\alpha\in \mathbb{N}^n\setminus
\{0\}$, since $R$ is $\Sigma$-semicommutative. Then
$\sigma^{\alpha}(a)\in r_R(SR)=eR$. By Lemma \ref{lema-Sigmaidem}
(2) we have that $\sigma^{\alpha}(e)=e$, thus
$\sigma^{\alpha}(ea)=\sigma^{\alpha}(e)\sigma^{\alpha}
(a)=e\sigma^{\alpha}(a)=\sigma^{\alpha}(a)$. Then $ea=a$ since
$\sigma^{\alpha}$ is a monomorphism, which implies that $a\in eR$.
Therefore $r_R(S)\subseteq eR$. In this way, $eR=r_R(S)$ and thus
$R$ is Baer.

\vspace{0.2cm}

(1) $\Rightarrow$ (4): Suppose that $R$ is Baer and let $H$ be a
nonempty subset of $A$. Let $H_0$ be the set of all coefficients of
elements of $H$. Notice that $H_0$ is a nonempty subset of $R$, thus
$l_R(H_0)=Re$, for some idempotent $e$ of $R$.

Let $h=\sum_{k=0}^tb_kX_k\in H$, with $X_k\in {\rm Mon}(A)$ and
$b_k\in R$, for $1\leq k\leq t$. Then we have that
$eh=\sum_{k=0}^t(eb_k)X_k=0$, since $e\in l_R(H_0)$ and $b_k\in H_0$
for each $k$. Thus $eH=0$, which implies that $AeH=0$. Therefore
$Ae\subseteq l_A(H)$.

For the other inclusion, let $0\neq g=\sum_{j=0}^mc_jY_j$ be a
nonzero polynomial in $l_A(H)$ and let $0\neq
h=\sum_{k=0}^tb_kX_k\in H$, where $c_j, b_k\in R$, $Y_j, X_k\in {\rm
Mon}(A)$. Then $gh=0$. Since $R$ is a $\Sigma$-semicommutative Baer
ring then by Corollary  \ref{cor-relatrigid} (3), $R$ is
$\Sigma$-skew Armendariz, so $c_j\sigma^{\alpha_j}(b_k)=0$, for each
$0\leq j\leq m$, $0\leq k\leq t$, where $\alpha_j={\rm exp}(Y_j)$.
Thus $c_j\in l_R(\sigma^{\alpha_j}(H_0))$. By Proposition
\ref{prop-anulad} (1) we have that
$l_R(\sigma^{\alpha_j}(H_0))=l_R(H_0)$. Then $c_j\in l_R(H_0)=Re$,
for $0\leq j\leq m$. Now, by Proposition \ref{prop-sigmaimplAbel} we
have that $Re=eR$, then $l_R(H_0)$ is an ideal of $R$. By
Proposition \ref{prop-anulad} (3), (4), we have that
$\sigma_i(l_R(H_0)\subseteq l_R(H_0)$ and
$\delta_i(l_R(H_0)\subseteq l_R(H_0)$ for $1\leq i\leq n$,
respectively. That is, $l_R(H_0)$ is a ($\Sigma, \Delta$)
\emph{invariant ideal} of $R$ \cite[Definition 2.1
(i)]{LezamaAcostaReyes}. Then by \cite[Proposition 2.6
(i)]{LezamaAcostaReyes} we have that
$l_R(H_0)A=\{f=\sum_{k=0}^ta_kX_k\in A\mid a_k\in l_R(H_0),\ 0\leq
k\leq t\}=l_R(H_0)\langle x_1,\dots, x_n\rangle$ is an ideal of $A$.
As $c_j\in l_R(H_0)$ then $c_jY_j\in l_R(H_0)A$ for $0\leq j\leq m$,
and so $g=\sum_{j=0}^mc_jY_j\in l_R(H_0)A$, since $l_R(H_0)A$ is an
ideal of $A$. Now, by Lemma \ref{lema-propidempot} (3) we have that
$Ae=eA$, then $l_R(H_0)A=ReA=RAe=Ae$. Therefore $g\in Ae$, that is,
$l_A(H)\subseteq Ae$. Thus $A$ is Baer.

\vspace{0.2cm}

(4) $\Rightarrow$ (3) It is immediate.

\vspace{0.2cm}

(3) $\Rightarrow$ (2) Suppose that $A$ is quasi-Baer and let $I$ be
a right ideal of $R$. Then $IA=\{f=\sum_{k=0}^ta_kX_k\in A\mid
a_k\in I,\ 0\leq k\leq t\}= I\langle x_1,\dots, x_n\rangle$ is a
right ideal of $A$. Since $A$ is quasi-Baer, $r_A(IA)=eA$, for some
idempotent element $e$ of $A$. Of course, $e\in R$, and thus by
Lemma \ref{lema-propidempot} (3),  $r_A(IA)=eA=Ae$ is an ideal of
$A$.

Notice that $r_R(I)=r_A(IA)\cap R$. Indeed, in general $r_A(IA)\cap
R\subseteq r_R(I)$. For the other inclusion, let $r\in r_R(I)$ and
$h=\sum_{j=0}^mb_jX_j\in IA$. Thus, $b_j\in I$ and therefore
$b_jr=0$, for $\ 0\leq j\leq m$. Then
\[
b_jr=0\Rightarrow \left\{
  \begin{array}{ll}
    b_jR\sigma^{\alpha}(r)=0,\  \alpha\in \mathbb{N}^n\setminus
\{0\}, & \text{ since } R \text{ is }
\Sigma-\text{semicommutative};\\
       b_j\delta^{\beta}(r)=0,\ \beta\in \mathbb{N}^n\setminus \{0\},& \text{ since } R \text{ is } \Delta-\text{compatible}.
  \end{array}
\right.
\]
Thus $b_jr=0$ implies $b_jw(r)=0$, for all $j$, where $w$ is any
composition of powers of $\sigma_i$ and $\delta_i$, $1\leq i\leq n$.
Then  $b_jX_jr=0$, since $b_jX_jr=\sum_{s=1}^l b_jw_s(r)X_{j_s}$ for
some $l\geq 1$, $w_s$ any composition of powers of $\sigma_i$ and
$\delta_i$, and $X_{j_s}\in {\rm Mon}(A)$. Thus
$\left(\sum_{j=0}^mb_jX_j\right)r=hr=0$ and so $r\in r_A(IA)$. Now,
$r_R(I)=r_A(IA)\cap R=eA\cap R=eR$. Therefore $R$ is quasi-Baer.
\end{proof}
%%%%%%%%%%%%%%%%%%%%%%%%
Since skew polynomial rings of injective type are examples of skew
PBW extensions, we obtain immediately the following result.
%%%%%%%%%%%%%%%%%%%%%%%%
\begin{corollary}\cite[Theorem 2.2]{Kim2014}\label{cor-Teo2.2Baser}
Let $R$ be $\sigma$-semicommutative ring with $\sigma$ an injective
endomorphism of $R$. Then the following statements are equivalent:
\begin{enumerate}
\item[\rm (1)] $R$ is Baer.
\item[\rm (2)] $R$ is quasi-Baer.
\item[\rm (3)] $R[x; \sigma]$ is quasi-Baer.
\item[\rm (4)] $R[x; \sigma]$ is Baer.
\end{enumerate}
\end{corollary}
%%%%%%%%%%%%%%%%%%%%%%%%
The following asserts that the condition of
$\Sigma$-semicommutativity of $R$ in Theorem \ref{teo-equivBaer} is
not superfluous.
%%%%%%%%%%%%%%%%%%%%%%%%
\begin{example}\label{ex-ex2Iranocont}
Let $R = \mathbb{K}[s, t]/\langle st\rangle$ and $A=R[x_1;
\sigma_1][x_2; \theta_2]=\sigma(R)\langle x_1,x_2\rangle$ be the
skew PBW extension as in Example \ref{ex-nosemicom}. Notice that
$l_R(\overline{s})=R\overline{t}$ is not generated by any idempotent
of $R$, hence $R$ is not quasi-Baer. Hirano in \cite[Example
2]{Hirano2001} showed that $R[x_1; \sigma_1]$ is quasi-Baer. Thus by
\cite[Theorem 1.2]{Birkenmeier2001} we have that $R[x_1;
\sigma_1][x_2; \theta_2]$ is quasi-Baer. Therefore $A =
\sigma(R)\langle x_1,x_2\rangle=R[x_1; \sigma_1][x_2; \theta_2]$ is
a quasi-Baer skew PBW extension  of a not quasi-Baer ring $R$. In
Example \ref{ex-ex2Irano} we saw that $R$ is not
$\Sigma$-semicommutative.
\end{example}
%%%%%%%%%%%%%%%%%%%%%%%%
The condition of $\Delta$-compatibility of $R$ in Theorem
\ref{teo-equivBaer} is not superfluous. Following Armendariz et al.
\cite[Example 2.10]{Armendarizetal1987}, let
$\frac{\mathbb{Z}_2[t]}{\langle t^{2}\rangle}$, where
$\mathbb{Z}_2[t]$ is the polynomial ring over the field
$\mathbb{Z}_2$ of two elements, and $\langle t^2\rangle$ is the
ideal of $\mathbb{Z}_2[t]$ generated by $t^2$. Consider the
derivation $\delta$ in $R$ defined by $\delta(\overline{t}) = 1$
where $\overline{t} = t+\langle t^2\rangle$ in $R$, and consider the
Ore extension $R[x; \delta] = \frac{\mathbb{Z}_2[t]}{\langle
t^{2}\rangle}[x;\delta]$. If we set $e_{11} = \overline{t}x,\ e_{12}
= \overline{t},\ e_{21} = \overline{t}x^2 + x$, and $e_{22} = 1 +
\overline{t}x$, then they form a system of matrix units in
$R[x;\delta]$. Notice that the centralizer of these matrix units in
$R[x;\delta]$ is $\mathbb{Z}_2[x^{2}]$, whence $R[y;\delta] \cong
M_2(\mathbb{Z}_2[x^{2}]) \cong M_2(\mathbb{Z}_2)[y]$. However,
$M_2(\mathbb{Z}_2)$ is not isomorphic to
$\frac{\mathbb{Z}_2[t]}{\langle t^{2}\rangle}$. Therefore $R[x;
\delta]$ is a Baer ring, and hence quasi-Baer, but $R$ is not
quasi-Baer. Of course, if $\Delta = \{\delta\}$, then $R[x;\delta]$
is a skew PBW extension of derivation type over $R$ where
$\overline{tt} = 0$ but $\overline{t}\delta(\overline{t}) =
\overline{t} \neq 0$, that is, $R$ is not $\Delta$-compatible.
%%%%%%%%%%%%%%%%%%%%%%%%
\begin{theorem}\label{teo-RSigmasemisiiASigmasemi}
Let $A = \sigma(R)\langle x_1,\dotsc, x_n\rangle$ be a skew PBW
extension over a Baer ring $R$. If the conditions established in
Proposition \ref{prop-indicsigma} hold, then $R$ is
$\Sigma$-semicommutative if and only if $A$ is
$\overline{\Sigma}$-semicommutative.
\end{theorem}
\begin{proof}
Consider the elements $f = \sum_{j=0}^t a_j X_j$, $g = \sum_{k=0}^m
b_kY_k\in A$ such that $fg=0$. By Corollary \ref{cor-relatrigid},
$R$ is $\Sigma$-skew Armendariz and $(\Sigma, \Delta$)-compatible,
and so \cite[Proposition 3.2]{ReyesSuarez2019Radicals} shows that
$R$ satisfies the condition (SA1). In this way, $fg=0$ implies
$a_jb_k=0$, for $0\leq j\leq t$ and $0\leq k\leq m$. Since $R$ is
$\Sigma$-semicommutative, $a_jR\sigma^{\alpha}(b_k)=0$, for
$\alpha\in \mathbb{N}^n\setminus \{0\}$, and having in mind that $R$
is $(\Sigma, \Delta$)-compatible, the equality
$a_jR\sigma^{\alpha}(b_k)=0$ implies
$a_jR\delta^{\theta}\sigma^{\alpha}(b_k)=0$, for $\alpha,\theta\in
\mathbb{N}^n\setminus \{0\}$. By Proposition \ref{prop-indicsigma}
we know that $\sigma_i\delta_j=\delta_j\sigma_i$, whence $0 =
a_jR\delta^{\theta}\sigma^{\alpha}(b_k)=a_jR\sigma^{\alpha}\delta^{\theta}(b_k)=a_jR\sigma^{\alpha}\delta^{\theta}\sigma^{\beta}\delta^{\gamma}=a_jRw(b_j)$,
where $w$ are compositions of powers of $\sigma_i$ and $\delta_i$,
$1\leq i\leq j$. Thus, $fh\overline{\sigma}^{\alpha}(g)=0$, for all
$h\in A$ and each $\alpha\in \mathbb{N}^n\setminus \{0\}$, i.e., $A$
is $\overline{\Sigma}$-semicommutative.

For the converse, notice that if $A$ is
$\overline{\Sigma}$-semicommutative, then for $a,b\in R$, $ab=0$
implies $aA\overline{\sigma}^{\alpha}(b)=0=aA\sigma^{\alpha}(b)$. In
particular, $aR\sigma^{\alpha}(b)=0$, for $\alpha\in
\mathbb{N}^n\setminus \{0\}$, that is, $R$ is
$\Sigma$-semicommutative.
\end{proof}
%%%%%%%%%%%%%%%%%%%%%%%%
\begin{example}
Let $R = \mathbb{K}[s, t]/\langle st\rangle$ and $A := R[x_1;
\sigma_1][x_2; \theta_2]=\sigma(R)\langle x_1,x_2\rangle$ be the
skew PBW extension as in  Example \ref{ex-nosemicom}. Since $R$ is
not $\Sigma$-semicommutative for $\Sigma=\{\sigma_1, \sigma_2\}$,
then $A$ is not $\overline{\Sigma}$-semicommutative for
$\overline{\Sigma}=\{\overline{\sigma_1}, \overline{\sigma_2}\}$,
where $\overline{\sigma_1}$ and $\overline{\sigma_2}$ are defined as
in Proposition \ref{prop-indicsigma}.
\end{example}
%%%%%%%%%%%%%%%%%%%%%%%%
%%%%%%%%%%%%%%%%%%%%%%%%
\section{Applications to topological conditions for skew PBW extensions}
Recall that a ring $R$ is called a \emph{pm ring} if every prime
ideal is contained in a unique maximal ideal. According to Kim et
al. \cite{Kim2006}, a ring is called \emph{nil-semisimple} if it has
no nonzero nil ideals. Nil-semisimple rings are semiprime, but
semiprime rings need not be nil-semisimple \cite[Example 1.2 and
Proposition 1.3]{Hwang2006}. Following Rowen \cite[Definition
2.6.5]{Rowen1991}, an ideal $P$ of a ring $R$ is called
\emph{strongly prime} if $P$ is prime and $R/P$ is nil-semisimple.
Maximal ideals are clearly strongly prime. A ring $R$ is called
\emph{weakly pm} if every strongly prime ideal of $R$ is contained
in a unique maximal ideal of $R$ \cite[p. 134]{Hwang2006}. Following
Goodearl and Warfield, \cite{GoodearlWarfield2004}, an ideal $P$ of
$R$ is called \emph{J-prime} if $P$ is a prime ideal and $J(R/I) =
0$. Clearly, every maximal ideal is J-prime. Now, a ring $R$ is
called \emph{J-pm} if every J-prime ideal is contained in a unique
maximal ideal \cite[p. 59]{Jiang2019}.

\vspace{0.2cm}

In this section, we use the following notation for a ring $R$: ${\rm
Spec}(R)$ is the set of all prime ideals of $R$; ${\rm SSpec}(R)$ is
the set of all strongly prime ideals of $R$; J-${\rm Spec}(R)$
denotes the set of all J-prime ideals of $R$; ${\rm Max}(R)$ is the
set of all maximal ideals of $R$. If $S$ is a subset of $R$, then
$W(S) := \left\{P\in {\rm Spec}(R)\mid S\nsubseteq P\right\}$, $O(S)
:= \left\{P\in {\rm SSpec}(R)\mid S\nsubseteq P\right\}$, $D(S) :=
\left\{P\in \text{J-} {\rm Spec}(R)\mid S\nsubseteq P\right\}$, and
write $W(a)$, $O(a)$, $D(a)$ in case $S=\{a\}$, respectively.

\vspace{0.2cm}

It is known that $\{W(I)\mid  I \text{ is an ideal of } R\}$ is a
topology for ${\rm Spec}(R)$, since $W(R) = {\rm Spec}(R)$; $W(0) =
{\rm Spec}(R)$; $W(I_1)\cap W(I_2)=W(I_1I_2)$ and $\bigcup_{j\in
J}W(I_j)=W\left(\sum_{j\in J}I_j\right)$, for $I_1, I_2$ and $I_j$
ideals of $R$. This topology is called \emph{Zariski} (or {\em
Stone} or {\em hull-kernel}) {\em topology}, and $\{W(a)\mid a\in
R\}$ is a base for this topology \cite[Lemma 3.1]{Shin1973}. Also,
the set $\{O(I)\mid I \text{ is an ideal of } R\}$ is a topology for
${\rm SSpec}(R)$ with basis $\{O(a)\mid a\in R\}$ \cite[Lemma
3.1]{Hwang2006}, and $\{D(I)\mid  I \text{ is an ideal of } R\}$ is
a topology for J-${\rm Spec}(R)$ with a basis $\{D(a)\mid a\in R\}$
\cite[Lemma 4.1]{Jiang2019}. Notice that ${\rm Spec}(R)$ is $T_0$
and compact \cite[Lemma 3.2]{Shin1973}; ${\rm Spec}(R)$ is normal if
and only if ${\rm Max}(R)$ is a retract of ${\rm Spec}(R)$ and ${\rm
Max}(R)$ is Hausdorff \cite[Theorem 1.6]{Sun1991}; and ${\rm
Max}(R)$ is compact and $T_1$ \cite[p. 185]{Sun1991}.
%%%%%%%%%%%%%%%%%%%%%%%%
\begin{proposition}\label{prop-generteo4.2Shin}
If $A$ is a skew PBW extension over a $\Sigma$-semicommutative Baer
ring $R$, then the following statements are equivalent:
\begin{enumerate}
\item[\rm (1)] $W(f)$ is closed ${\rm Spec}(A)$,  for each $f\in A$.
\item[\rm (2)] ${\rm Spec}(A)$ completely regular and Hausdorff.
\item[\rm (3)] ${\rm Spec}(A)$ is $T_1$.
\end{enumerate}
\end{proposition}
\begin{proof}
By assumption $R$ is a $\Sigma$-semicommutative Baer ring, so
Corollary \ref{cor-relatrigid} (2) implies that $A$ is reduced, that
is, 2-primal. By \cite[Theorem 4.2]{Shin1973} we can assert the
equivalence of the assertions.
\end{proof}
%%%%%%%%%%%%%%%%%%%%%%%%
Sun \cite[Theorem 1.6]{Sun1991} showed that ${\rm Spec}(R)$ is
normal if and only if ${\rm Max}(R)$ is a retract of ${\rm Spec}(R)$
and ${\rm Max}(R)$ is Hausdorff, for a ring $R$. Sun also studied
certain relationships between topological properties of ${\rm
Spec}(R)$ of a 2-primal ring $R$ with certain properties of ring
theory such as being pm.
%%%%%%%%%%%%%%%%%%%%%%%%
\begin{proposition}\label{prop-generteo2.3Sun}
If $A$ is a skew PBW extension over a $\Sigma$-semicommutative Baer
ring $R$, then the following affirmations are equivalent:
\begin{enumerate}
\item[\rm (1)] $A$ is pm.
\item[\rm (2)] ${\rm Max}(A)$ is a retract of ${\rm Spec}(A)$.
\item[\rm (3)] ${\rm Spec}(A)$ is normal (not necessary $T_1$), and hence we obtain the Hausdorffness of ${\rm Max}(A)$.
\end{enumerate}
\end{proposition}
\begin{proof}
Since $R$ is a $\Sigma$-semicommutative Baer ring, Corollary
\ref{cor-relatrigid} (2) shows that $A$ is reduced and therefore
2-primal. The result follows from \cite[Theorem 2.3]{Sun1991}.
\end{proof}
%%%%%%%%%%%%%%%%%%%%%%%%
Sun \cite[Theorem 3.7]{Sun1991} showed that the assignment to each
weakly symmetric ring of the maximal ideal space ${\rm Max}(R)$ is
functorial on the category of 2-primal rings and ring homomorphisms.
On the other hand, Hwang et al. \cite{Hwang2006} showed that for a
ring $R$, ${\rm SSpec}(R)$ is compact and ${\rm Max}(R)$ is a
compact $T_1$-space \cite[Lemma 3.2]{Hwang2006}, and ${\rm
SSpec}(R)$ is normal if and only if ${\rm Max}(R)$ is a retract of
${\rm SSpec}(R)$ and ${\rm Max}(R)$ is Hausdorff \cite[Proposition
3.6]{Sun1991}. The equivalent statements given by Sun \cite[Theorem
2.3]{Sun1991} were generalized by Hwang et al. \cite[Theorem
3.7]{Hwang2006} for NI rings, changing the rings pm and ${\rm
Spec}(R)$ by weakly pm and ${\rm SSpec}(R)$, respectively. The next
proposition extends these properties to the setting of skew PBW
extensions.
%%%%%%%%%%%%%%%%%%%%%%%%
\begin{proposition}\label{prop-generteo3.7Hwang}
Let $A = \sigma(R)\langle x_1,\dotsc, x_n\rangle$ be a skew PBW
extension over $R$. If one of the following conditions holds,
\begin{enumerate}
\item[(i)] $R$ is $\Sigma$-semicommutative and reduced, or
\item[(ii)] $R$ is $\Sigma$-semicommutative and Baer, or
\item[(iii)] $R$ is NI and weak $(\Sigma, \Delta)$-compatible,
\end{enumerate}
then the following affirmations are equivalent:
\begin{enumerate}
\item[(1)] $A$ is weakly pm.
\item[(2)] ${\rm SSpec}(A)$ is normal.
\item[(3)] ${\rm Max}(A)$ is a retract of ${\rm SSpec}(A)$.
\end{enumerate}
\end{proposition}
\begin{proof}
In each of the cases, we will see that the skew PBW extension $A$ is
NI, and hence the equivalence of the statements follows from
\cite[Theorem 3.7]{Hwang2006}.
\begin{enumerate}
\item[(i)] If $R$ is $\Sigma$-semicommutative and reduced, then by Theorem \ref{teo-compilSigma-semic} $A$ is reduced and therefore NI.
\item[(ii)] If $R$ is  $\Sigma$-semicommutative and Baer, then Corollary \ref{cor-relatrigid} (2) implies that $A$ is reduced, and so NI.
\item[(iii)] If $R$ is NI and weak $(\Sigma, \Delta)$-compatible, from \cite[Theorem 3.3]{SuarezChaconReyes2021} we conclude that $A$ is NI.
\end{enumerate}
\end{proof}
Jiang et al. \cite{Jiang2019} showed that for a ring $R$, J-${\rm
Spec}(R)$ is a compact space, and if J-${\rm Spec}(R)$ is normal,
then ${\rm Max}(R)$ is Hausdorff (see \cite[Lemma 4.3]{Jiang2019}).
Also, they proved that if ${\rm Max}(R)$ is a retract of J-${\rm
Spec}(R)$, then $R$ is J-pm (see \cite[Proposition 4.6]{Jiang2019}).
An analog of \cite[Theorem 3.7]{Hwang2006} was given by Jiang et al.
in \cite[Theorem 4.10]{Jiang2019}, changing weakly pm rings by J-pm
rings, NI by NJ and ${\rm SSpec}(R)$ by J-${\rm Spec}(R)$. The
following proposition is analogous to \cite[Theorem 4.10]{Jiang2019}
for the setting of skew PBW extensions.
%%%%%%%%%%%%%%%%%%%%%%%%
\begin{proposition}\label{prop-generteo4.10Jian}
If $A = \sigma(R)\langle x_1,\dotsc, x_n\rangle$ is a skew PBW
extension over a $\Sigma$-semicommutative reduced ring $R$, then the
following statements are equivalent:
\begin{enumerate}
\item[(1)] $A$ is J-pm.
\item[(2)] J-${\rm Spec}(A)$ is normal.
\item[(3)] ${\rm Max}(A)$ is a retract of J-${\rm Spec}(A)$.
\end{enumerate}
\end{proposition}
\begin{proof}
If $R$ is $\Sigma$-semicommutative and  reduced then $R$
$(\Sigma,\Delta)$-compatible and $A$ is reduced, by Theorem
\ref{teo-compilSigma-semic}. Now, $R$ is 2-primal, since $R$ is
reduced. Then by \cite[Theorem 4.11]{LouzariReyes2020}, $J(A) =
N_{*}(A)$. Since $A$ is reduced, $N_{*}(A)=N(A)=\{0\}$. Thus, $J(A)
= N(A)$, i.e., $A$ is an NJ ring. The result follows from
\cite[Theorem 4.10]{Jiang2019}.
\end{proof}
%%%%%%%%%%%%%%%%%%%%%%%%
\begin{corollary}\label{cor-HausdNJ}
If $A = \sigma(R)\langle x_1,\dotsc, x_n\rangle$ is an NJ skew PBW
extension which satisfies one of the equivalent statements of
Proposition \ref{prop-generteo4.10Jian}, then ${\rm Max}(R)$ is
Hausdorff.
\end{corollary}
\begin{proof}
The assertion follows from \cite[Proposition 4.8]{Jiang2019}.
\end{proof}
%%%%%%%%%%%%%%%%%%%%%%%%
\begin{corollary}\label{cor-NJ impl NI}
If $A = \sigma(R)\langle x_1,\dotsc, x_n\rangle$ is a skew PBW
extension over a $\Sigma$-semicommutative reduced ring $R$, then the
following affirmations are equivalent:
\begin{enumerate}
\item[(1)] $A$ is weakly pm;
\item[(2)]  ${\rm SSpec}(A)$ is normal;
\item[(3)] ${\rm Max}(A)$ is a retract of ${\rm SSpec}(A)$.
\end{enumerate}
\end{corollary}
\begin{proof}
By the proof of Proposition \ref{prop-generteo4.10Jian} we have that
$A$ is an NJ ring. Thus $N(A)$ is an ideal of $A$ and therefore $A$
is NI. The result follows then from \cite[Theorem 3.7]{Hwang2006}.
\end{proof}
%%%%%%%%%%%%%%%%%%%%%%%%
\begin{corollary}\label{cor-combinspect}
Let $A = \sigma(R)\langle x_1,\dotsc, x_n\rangle$ be a skew PBW
extension of derivation type over a ring $R$. If one of the
following conditions holds,
\begin{enumerate}
\item[(i)] $R$ is $\Sigma$-semicommutative and reduced,
\item[(ii)] $R$ is $\Sigma$-semicommutative and Baer, or
\item[(iii)] $R$ is NI and weak $(\Sigma, \Delta)$-compatible,
\end{enumerate}
then the following statements are equivalent:
\begin{enumerate}
\item[(1)] $A$ is J-pm.
\item[(2)] J-${\rm Spec}(A)$ is normal.
\item[(3)] ${\rm Max}(A)$ is a retract of J-${\rm Spec}(A)$.
\end{enumerate}
\end{corollary}
\begin{proof}
According to Proposition \ref{prop-generteo3.7Hwang}, if $R$
satisfies any of the give conditions, then $A$ is an NI ring. Thus,
\cite[Theorem 4.10]{SuarezChaconReyes2021} implies that $A$ is NJ,
whence the statements are equivalent by \cite[Theorem
4.10]{Jiang2019}.
\end{proof}

\end{document}